\documentclass[a4paper,11pt]{article}
\usepackage[normalem]{ulem}
\usepackage{amsmath,amsthm,amssymb,enumerate}
\usepackage[T1]{fontenc}
\usepackage{gensymb}
\usepackage[square,sort&compress,comma,numbers]{natbib}
\usepackage[sc]{mathpazo}
\usepackage[mathscr]{euscript}
\usepackage{multirow} 
\usepackage{newtxtext,newtxmath}
\usepackage{subfig}
\usepackage{graphicx}
\usepackage{epstopdf}
\usepackage{cancel}
\usepackage[margin=3cm]{geometry}
\usepackage{tikz}
\usepackage{caption}
\usetikzlibrary{shapes,calc}
\usepackage{verbatim}
\usepackage{mathrsfs}
\usepackage{accents}
\usepackage[utf8]{inputenc}
\usepackage{lipsum}

\usetikzlibrary{decorations.pathmorphing}
\usetikzlibrary{decorations.pathreplacing}
\usetikzlibrary{positioning}
\usetikzlibrary{shapes}
\usetikzlibrary{arrows}
\usetikzlibrary{patterns}
\usetikzlibrary{fadings}
\usetikzlibrary{plotmarks}
\usetikzlibrary{calc}
\usetikzlibrary{intersections}
\usepackage{paralist}
\usepackage{bbm}
\usepackage{latexsym}           
\usepackage{enumerate}
\usepackage{enumitem}

\setlist{noitemsep, topsep=0.8ex, partopsep=0pt
, leftmargin=3em}
\setlist[1]{labelindent=\parindent}

\newlist{axioms}{enumerate}{1}
\setlist[axioms]{font=\bfseries}

\newlist{alphenum}{enumerate}{1}
\setlist[alphenum]{label=\textbf{(\alph*)}, leftmargin=4em}

\newlist{alphienum}{enumerate}{1}
\setlist[alphienum]{label=\textit{(\alph*)}}

\newlist{romanenum}{enumerate}{1}
\setlist[romanenum]{label=\textit{(\roman*)}}
\newtheorem*{Proof of}{\it Proof of Theorem \ref{th:existence:contslevel}}

\newlist{romaninenum}{enumerate*}{1}
\setlist[romaninenum]{label=\textit{(\roman*)}}
\usepackage[vlined]{algorithm2e}
\SetKwIF{If}{ElseIf}{Else}{if}{}{else if}{else}{endif}
\SetKwFor{For}{for}{}{endfor}
\usepackage[noabbrev, capitalise]{cleveref}
\usepackage{tabu}
\crefname{equation}{\unskip}{\unskip}
\creflabelformat{equation}{#2(#1)#3}
\usepackage[vlined]{algorithm2e}

\SetFuncSty{textsc}
\SetKwFunction{frun}{Run}
\SetKwHangingKw{arun}{\frun}
\SetKwFunction{fset}{Set}
\SetKwHangingKw{aset}{\fset}
\SetKwFunction{ffind}{Find}
\SetKwHangingKw{afind}{\ffind}
\SetKwFunction{fselect}{Select}
\SetKwHangingKw{aselect}{\fselect}
\SetKwFunction{fcompute}{Compute}
\SetKwHangingKw{acompute}{\fcompute}
\SetKwFunction{fsolve}{Solve}
\SetKwHangingKw{asolve}{\fsolve}
\SetKwFunction{fupdate}{Update}
\SetKwHangingKw{aupdate}{\fupdate}
\SetKwFunction{festimate}{Estimate}
\SetKwHangingKw{aestimate}{\festimate}
\SetKwFunction{fmark}{Mark}
\SetKwHangingKw{amark}{\fmark}
\SetKwFunction{fbreak}{Break}
\SetKwHangingKw{abreak}{\fbreak}
\SetKwFunction{frefine}{Refine}
\SetKwHangingKw{arefine}{\frefine}
\SetKwFunction{compute}{Compute}
\SetKwFunction{set}{Set}

{\algorithm}%
{\endalgorithm}

\newtheorem{thm}{Theorem}[section]

\newtheorem{lem}[thm]{Lemma}

\theoremstyle{definition}

\theoremstyle{remark}
\newtheorem{rem}{Remark}[section]

\numberwithin{equation}{section}

\newcommand{\cE}{\mathcal E}

\newcommand{\cV}{\mathcal V}
\newcommand{\cT}{\mathcal T}

\newcommand{\vket}{von K\'{a}rm\'{a}n equations}
\newcommand{\vk}{von K\'{a}rm\'{a}n}

\newcommand{\fl}{\;\text{ for all }}

\newcommand{\half}{\frac{1}{2}}
\newcommand{\trinl}{\ensuremath{|\!|\!|}}
\newcommand{\trinr}{\ensuremath{|\!|\!|}}

\newcommand{\ds}{{\rm\,ds}}

\newcommand{\M}{\text{M}}
\newcommand{\pw}{\text{pw}}

\newcommand{\T}{\mathcal{T}}

\def \R{{{\Bbb R}}}

\allowdisplaybreaks

\def\R{\mathbb{R}}

\def\O{\Omega}

\def\cV{\mathcal{V}}

\def\cT{\mathcal{T}}
\def\cE{{\mathcal{E}}}

\setlist[itemize,2]{label={$\star$}}

\title{{Morley Finite Element Method for the \\von K\'{a}rm\'{a}n Obstacle Problem}}

\author{
	C. Carstensen\footnote{Department of Mathematics, HU Berlin, 10099 Berlin, Germany. Distinguished Visiting Professor, Department of Mathematics, IIT Bombay, Powai, Mumbai 400076, India. (cc@math.hu-berlin.de),}, 
S. Gaddam, N. Nataraj, A. K. Pani{\footnote{Department of Mathematics, IIT Bombay, Powai, Mumbai 400076, India. ({gsharat,neela,akp}@math.iitb.ac.in),}} \, and D. Shylaja\footnote{IITB-Monash Research Academy, IIT Bombay, Powai, Mumbai 400076, India. (devikas@math.iitb.ac.in).}
}

\date{\today}
\begin{document}
\maketitle
\begin{abstract}
This paper focusses on the \vket\, for the moderately large deformation of a very thin plate with the convex obstacle constraint leading to a coupled system of semilinear fourth-order obstacle problem and motivates its nonconforming Morley finite element approximation. The first part establishes the well-posedness of the \vk\, obstacle problem and also discusses the uniqueness of the solution under an {\em a priori} and an {\em a posteriori} smallness condition on the data. The second part of the article discusses the regularity result of Frehse from 1971 and combines it with the regularity of the solution on a polygonal domain. The third part of the article shows an {\em a priori} error estimate for optimal convergence rates for the Morley finite element approximation to the von K\'{a}rm\'{a}n obstacle problem for small data. The article concludes with numerical results that illustrates the requirement of smallness assumption on the data for optimal convergence rate.
\end{abstract}

\section{Introduction}
{\bf Short history of related work.} The \textbf{\vket}\,\cite{ciarlet1997mathematical} model the bending of very thin elastic plates through a system of fourth-order semi-linear elliptic equations; cf. \cite{berger1968karman, ciarlet1997mathematical, knightly1967existence} and references therein for the existence of solutions, regularity, and bifurcation phenomena. The papers \cite{brezzi1978finite, miyoshi1976mixed, quarteroni1979hybrid, reinhart1982numerical,mallik2016nonconforming, brenner2017c, carstensen2018priori,carstensen2017nonconforming} study the approximation and error bounds for regular solutions to \vket\, using conforming, mixed, hybrid, Morley, $C^0$ interior penalty and discontinuous Galerkin finite element methods (FEMs). 


\medskip
The \textbf{obstacle problem} is a prototypical example for a variational inequality and arises in contact mechanics, option pricing, and fluid flow problems. The location of the free boundary is not known {\em a priori} and forms a part of the solution procedure. For the theoretical and numerical aspects of variational inequalities, see \cite{glowinski2008lectures, kinderlehrer1980introduction}. A unified convergence analysis for the fourth-order linear two-sided obstacle problem of clamped Kirchhoff plates in \cite{brenner2012finite,brenner2013morley,brenner2012quadratic} studies $C^1$ FEMs, $C^0$ interior penalty methods, and classical nonconforming FEMs on convex domains and, analyse the $C^0$ interior penalty and the Morley FEM on polygonal domains. 

\medskip
The \textbf{obstacle problem for \vket}\, with a nonlinearity together with a free boundary offers additional difficulties. The obstacle problem in \cite{muradova2007unilateral,yau1992obstacle,miersemann1992stability} concerns a different plate model with continuation, spectral, and complementarity methods, while the papers \cite{ohtake1980analysisI, ohtake1980analysisII} study conforming penalty FEM.


\medskip
The present paper is the first on the fourth-order semilinear obstacle problem of a (very thin) \vk\, plate. The article derives existence, uniqueness (under smallness assumption on data) and regularity results of the \vk\, obstacle problem. 
Nonconforming FEMs appear to be more attractive than the classical $C^1$ conforming FEMs, so this article suggests the Morley FEM to approximate the \vk\, obstacle problem and derives an optimal order {\em a priori} error estimate with the best approximation plus a linear perturbation. 


\medskip
\noindent
{\bf Problem Formulation.} Given an obstacle $\chi \in H^2(\O)$ with $\max\chi(\partial \Omega):=\max_{x\in\partial \Omega}\chi(x) <0$, define the non-empty, closed, and convex subset  
$$K:=\{\varphi \in H^2_0(\O): \varphi \ge \chi \mbox{ a.e. in } \Omega\}$$
of $H^2_0(\O)$ in a bounded polygonal domain $\O\subset\R^2.$ The Hessian $D^2$ and \vk\, bracket $[\varphi_1, \varphi_2]:=\varphi_{1xx}\varphi_{2yy}+\varphi_{1yy}\varphi_{2xx}-2\varphi_{1xy}\varphi_{2xy}$ with partial derivatives $(\bullet)_{xy}:=\partial^2 (\bullet)/\partial x\partial y$ etc. Define for $ \varphi_1, \varphi_2, \varphi_3 \in H^2_0(\O)$ the weak forms
\begin{align}\label{defnaandb}
a(\varphi_1,\varphi_2):=(D^2\varphi_1,D^2\varphi_2)_{L^2(\O)} \,{\rm{ and }}\,\,
b(\varphi_1,\varphi_2,\varphi_3):=-\frac{1}{2}([\varphi_1,\varphi_2],\varphi_3)_{L^2(\O)}
\end{align} 
with the $L^2(\O)$ inner product $(\bullet,\bullet)_{L^2(\O)}$. It is well established \citep[Corollary 2.3]{brenner2017c} and follows from symmetry of the \vk\, bracket $[\bullet,\bullet]$ that $b:H^2_0(\O)^3\rightarrow\R$ is symmetric with respect to all the three arguments. The weak formulation of the \vk\, obstacle problem seeks $(u,v) \in K \times H^2_0(\O)$ such that
\begin{subequations}\label{wform}
	\begin{align}
	& a(u,u-\varphi_1)+ 2b(u,v,u-\varphi_1)\le (f,u-\varphi_1)_{L^2(\O)}   \; \; \fl \varphi_1\in K\label{wforma},\\
	& a(v,\varphi_2)-b(u,u,\varphi_2)   =0       \; \;     \fl \varphi_2 \in H^2_0(\O)\label{wformb}.
	\end{align}
\end{subequations}

\medskip
\noindent
{\bf Results and overview.} A smallness assumption on the data is derived in Section 2 to show that \eqref{wform} is well-posed. The regularity results of Section 3 establish that any solution $(u,v)$ to \eqref{wform} satisfies $u,v\in H^2_0(\O)\cap H^{2+\alpha}(\O)\cap C^2(\O)$ for the index $1/2<\alpha\le 1$ with $\alpha=\min\{\alpha',1\}$ and the index $\alpha'$ of elliptic regularity \cite{blum1980boundary} of the biharmonic operator in a polygonal domain $\O$. Section 4 introduces the Morley FEM and discusses the well-posedness of the discrete problem with an {\em a priori} and an {\em a posteriori} smallness condition on the data for global uniqueness. Section 5 derives {\em a priori} energy norm estimates of optimal order $\alpha$ for the Morley FEM under the smallness assumption on the data that guarantees global uniqueness of the minimizer on the continuous level. The article concludes with numerical results that illustrates the requirement of smallness assumption on the data for optimal convergence rate.


\medskip
\noindent
{\bf Notation.} Standard notation on Lebesgue and Sobolev spaces and their norms apply throughout the paper. For $s>0$ and $1 \le p \le \infty$, let $|\bullet|_{s}$ and $\|\bullet\|_{s}$ (resp. $|\bullet|_{s,p}$ and $\|\bullet\|_{s,p}$ ) denote the semi-norm and norm on $H^{s}(\Omega)$ (resp. $W^{s,p} (\Omega)$); $\|\bullet\|_{-s}$ denotes the norm in $H^{-s}(\Omega)$. The standard $L^2$ inner product and norm are denoted by $(\bullet,\bullet)_{L^2(\Omega)}$ and $\|\bullet\|_{L^2(\Omega)}$. The triple norm $\trinl\bullet\trinr :=\|D^2 \bullet \|_{L^2(\Omega)}$ is the energy norm defined by the Hessian and $\trinl \bullet \trinr_{\pw}:= \|D^2_{\pw} \bullet \|_{L^2(\Omega)}$ is its piecewise version with the piecewise Hessian $D^2_{\pw}$, $[\bullet,\bullet]_{\rm pw}$ denotes the piecewise version of the \vk\, bracket $[\bullet,\bullet]$ with respect to an underlying (non-displayed) triangulation. $H^{-2}(\O)$ is the dual space of the Hilbert space $(H^2_0(\O),\trinl \bullet \trinr)$. The elliptic regularity index $1/2<\alpha\le 1$ is determined by the interior angles of the domain $\O$ \cite{blum1980boundary} and is the same throughout this paper. The notation $A \lesssim B$ abbreviates $A\leq CB$ for some positive generic constant $C$ which depends on $\trinl u\trinr, \trinl v\trinr, \| u\|_{2+\alpha}, \| v\|_{2+\alpha}, \|f\|_{L^2(\Omega)}$; $A\approx B$ abbreviates $A\lesssim B\lesssim A$.
\newpage
\section{Well-posedness}\label{sec.wellposed}
This section establishes the well-posedness of the problem \eqref{wform}. 
The existence of a solution to \eqref{wform} follows with the direct method in the calculus of variations. The subsequent bound applies often in this paper and is based on Sobolev embedding. Let $C_{\rm S}$ denote the Sobolev constant in the Sobolev embedding $H^2_0(\O)\hookrightarrow C(\overline{\Omega})$ and let $C_{\rm F}$ denote the Friedrichs constant with \begin{align}\label{ctsembedding}
\|v\|_{L^{\infty}(\Omega)}\le C_{\rm S}\|v\|_{H^2_0(\O)}\quad \text{and}\quad\|v\|_{L^{2}(\Omega)}\le C_{\rm F}\trinl v\trinr \text{  for all  } v\in H^2_0(\O).
\end{align}
\begin{lem}[Bound for $b(\bullet,\bullet,\bullet)$ \cite{brezzi1978finite}]\label{Bound b}\noindent The trilinear form $b(\bullet,\bullet,\bullet)$ from \eqref{defnaandb} satisfies, for all $ \varphi_1, \varphi_2, \varphi_3 \in H^2_0(\O)$, that $b(\varphi_1,\varphi_2,\varphi_3)\le \trinl\varphi_1 \trinr\trinl\varphi_2 \trinr\| \varphi_3\|_{L^\infty(\O)}\le C_{\rm S} \trinl\varphi_1 \trinr\trinl\varphi_2 \trinr\trinl \varphi_3\trinr.$
\end{lem}
\noindent{For all $\xi\in H^2_0(\O),$ Lemma \ref{Bound b} implies $b(\xi,\xi,\bullet)\in H^{-2}(\O)$. 
Define  $G:H^2_0(\O)\rightarrow H^2_0(\O)$ by 
\begin{align}\label{defnG}
a(G(\varphi),\psi)=b(\varphi,\varphi,\psi)\text{  for all  } \varphi,\psi\in H^2_0(\O).
\end{align} 
This means $G(\xi)$ is the Riesz representation of the linear bounded functional $b(\xi,\xi,\bullet)$ in the Hilbert space $(H^2_0(\O),a(\bullet,\bullet))$. 
Consider the minimizer $u$ of the functional $j(\xi)$ for $\xi\in K$ and 
\begin{align}
j(\xi):= \half\trinl \xi\trinr^2\: +\half \trinl G(\xi)\trinr^2\: -(f,\xi)_{L^2(\O)}. \label{minprb}
\end{align}

\noindent\\
The equivalence of \eqref{wform} with \eqref{minprb}, for $K=H^2_0(\O)$, is established in \cite[Theorem 5.8.3]{ciarlet1997mathematical}. Analogous arguments also establish the equivalence, for any non-empty, closed, and convex subset $K$ of $H^2_0(\O)$, so the proof is omitted. This implies  that, to prove the existence of a solution to \eqref{wform}, it is sufficient to prove the existence of a minimizer to \eqref{minprb}.

\begin{thm}[\rm Existence]\label{th:existence:contslevel}
Given $(f,\chi)\in L^2(\O)\times H^2(\O)$ with $\max \chi(\partial\Omega)<0$, there exists a minimizer of $j(\bullet)$ in $K$
; each minimizer $u$ and $v:=G(u)$ solves \eqref{wform}.
\end{thm}
\begin{proof}
Given $\xi \in K$, the definition of $j(\bullet)$ in \eqref{minprb} and the Cauchy-Schwarz inequality lead to
$$\trinl \xi\trinr^2+ \trinl G(\xi)\trinr^2-2\|f\|_{-2}\trinl \xi\trinr \le 2j(\xi).
$$
This implies the lower bound
$$-\infty<-\|f\|^2_{-2}=\min_{t\ge 0}\big(t^2-2t\|f\|_{-2}\big)\le 2j(\xi) \fl\xi \in K.$$
Consequently, there exists a sequence $(u_n)_{n \in \mathbb{N}}$ in $K$ such that
$$j(u_n) \rightarrow \beta:=\inf_{\xi \in K}j(\xi)\in \R.$$
The Cauchy-Schwarz and the Young inequalities lead to
$$\trinl u_n\trinr^2+ 2\trinl G(u_n)\trinr^2-4\|f\|_{-2}^2 \le 2\trinl u_n\trinr^2+2\trinl G(u_n)\trinr^2-4\|f\|_{-2}\trinl u_n\trinr \\
\le 4j(u_n).$$
Consequently, $\trinl u_n\trinr^2+ 2\trinl G(u_n)\trinr^2\le 4j(u_n)+4\|f\|_{-2}^2 $. Since $j(u_n)$ is convergent, the sequences $(u_n)_{n \in \mathbb{N}}$ and $(G(u_n))_{n \in \mathbb{N}}$ are bounded in $H^2_0(\O)$. Hence, there exist $u, w \in H^2_0(\O)$ and a weakly convergent subsequence $(u_{n_k})_{k \in \mathbb{N}}$ such that
$$u_{n_k}\rightharpoonup u \text{ and } G(u_{n_k})\rightharpoonup w \text{ weakly in } H^2_0(\O)\text{ as } k\rightarrow\infty.$$
The non-empty closed convex set $K$ of $H^2_0(\O)$ is sequentially weakly closed and so $u \in K$. Since $u_{n_k}$ converges weakly to $u$ in $H^2_0(\O),$ this implies
\begin{equation}\label{j.f}
\int_{\Omega}fu_{n_k} \rightarrow \int_{\Omega}fu \mbox{ as } k \rightarrow \infty. \end{equation}
The compact embedding of $H^2_0(\O)$ in $L^{2}(\O)$ implies $u_{n_k} \rightarrow u$ in $L^{2}(\Omega).$ Further for a given $\varphi \in \mathcal{D}(\O)$, the definition of $G(\bullet)$ in \eqref{defnG}, the symmetry of $b(\bullet,\bullet, \bullet)$ with respect to second and third arguments, and the weak convergence of $u_{n_k}\rightharpoonup u$ in $H^2_0(\O)$ lead to
\begin{align*}
a(G(u_{n_k}),\varphi)=b(u_{n_k},\varphi,u_{n_k}) &\rightarrow b(u,\varphi,u)=a(G(u),\varphi)\text{ as } k\rightarrow\infty.
\end{align*}
Since $\varphi$ is arbitrary in the dense set $\mathcal{D}(\O)$ of $H^2_0(\O)$, this means $G(u_{n_k}) \rightharpoonup G(u)$ weakly in $H^2_0(\O)$ as $k \rightarrow \infty$. 
The sequentially weak lower semi-continuity of the norm $\trinl \bullet\trinr$ shows $j(u) \le \liminf_k j(u_{n_k}).$ This and $\lim_{k \rightarrow \infty}j(u_{n_k})=\beta \le\liminf_k j(u_{n_k})$ prove that $u$ minimizes $j$ in $K.$ By the definition of $G(\bullet)$ in \eqref{defnG}, $(u,G(u))$ solves \eqref{wform}. This concludes the proof.
\end{proof}

\noindent Theorem \ref{thm.contsdep} establishes an {\em a priori} bound and the uniqueness of the solution to \eqref{wform}. Recall the Sobolev (resp. Friedrichs) constant $C_{\rm S}$(resp. $C_{\rm F}$) from \eqref{ctsembedding}. 

	\begin{thm}[{\em a priori} bound and uniqueness] \label{thm.contsdep}
		Given $(f,\chi)\in L^2(\O)\times H^2(\O)$ with $\max \chi(\partial\Omega)<0$, there exists a positive constant $C(\chi)$ that depends only on $\chi$ such that any solution $(u,v)$ to \eqref{wform} satisfies $(a)$-$(b).$\\
		$(a)$ $\frac{1}{2}\trinl u\trinr^2+\trinl v\trinr^2\le N^2(f,u):= 2j(u)+2C^2_{\rm F}\|f\|^2 \le M^2(f,\chi):= C(\chi)+3C^2_{\rm F}\|f\|^2_{L^2(\O)}$.\\
		$(b)$ If $C^2_{\rm S}(\frac{1}{2}\trinl u\trinr^2+\trinl v\trinr^2)< (\sqrt{2}-1)^2$, then $(u,v)$ is the only solution to \eqref{wform}. 
	\end{thm}
	\begin{proof}
		Since $u$ is the minimizer of \eqref{minprb}, the Young inequality implies, for any $\varphi\in K$, that
		\begin{align*}
		\trinl u\trinr^2+\trinl G(u)\trinr^2&\le 2j(\varphi)+2(f,u)_{L^2(\O)}\le 2j(\varphi)+2C^2_{\rm F}\|f\|^2_{L^2(\O)}+\frac{1}{2}\trinl u\trinr^2.
		\end{align*}
		This proves for the minimizer $u$ of $j(\bullet)$ that
		\begin{align}
		\frac{1}{2}\trinl u\trinr^2+\trinl G(u)\trinr^2&\le 2j(u)+2C^2_{\rm F}\|f\|^2_{L^2(\O)}:=N^2(f,u)\le N^2(f,\varphi).\label{apostest}
		\end{align}
Since $\max \chi(\partial \Omega) <0$, $\{\chi\ge0\}:=\{x\in\Omega:\chi(x)\ge0\}$ is a compact subset of $\Omega$ and there exists an open set $\Omega_+$ around $\{\chi\ge0\}$ such that $\overline{\Omega_+}$ is a compact subset of $\Omega$. Consider the cut-off function $\psi \in \mathcal{D}(\Omega)$ such that $0\le \psi\le1$, $\psi=1$ in $ \{\chi\ge0\},$ ${\rm{supp}}(\psi)\subset \Omega_+,$ and define $\varphi:=\chi\psi\in H^2_0(\O)$. Then, $\varphi \ge \chi$ in $\Omega$, and so $\varphi \in K$. The construction of $\varphi$ ensures that 

\begin{equation}\label{boundvarphi}
\trinl \varphi\trinr_{}=| \varphi |_{H^2( \Omega_+)}\le C(\psi)\|\chi\|_{H^2( \Omega_+)}.
\end{equation}
This inequality, the definition of $G(\bullet),$ and Lemma \ref{Bound b} lead to $$\trinl G(\varphi)\trinr^2_{}=b(\varphi,\varphi,G(\varphi))\le C_{\rm S}\trinl \varphi\trinr^2_{}\trinl G(\varphi)\trinr_{}\le C_{\rm S}C^2(\psi)\trinl G(\varphi)\trinr_{}\|\chi\|^2_{H^2(\Omega_+)}.$$ Consequently, $\trinl G(\varphi)\trinr_{}\le C_{\rm S}C^2(\psi)\|\chi\|^2_{H^2(\Omega_+)}$. An application of the bounds for $\varphi$ and $G(\varphi)$ in \eqref{apostest} concludes the proof of final estimate of part $(a)$ with $C(\chi):=2C^2(\psi)\|\chi\|^2_{H^2(\Omega_+)}+C_{\rm S}^2C^4(\psi)\|\chi\|^4_{H^2(\Omega_+)}$ and $u$ being the minimizer of $j(\bullet)$ implies $N(f,u)\le N(f,\psi\chi)\le M(f,\chi).$ 
\noindent\\
To prove $(b)$, recall the definition of $G(\bullet)$ from \eqref{defnG} and let $(u_1,G(u_1))$ and $(u_2,G(u_2))$ be two solutions to \eqref{wform}. Set $e=u_1-u_2,\,\delta=G(u_1)-G(u_2)$, and choose $u=u_1$, $\varphi_1=u_2$ (respectively, $u=u_2$, $\varphi_1=u_1$) in \eqref{wforma} and add the resulting inequalities to deduce that
\begin{align}
\frac{1}{2}\trinl e \trinr^2\le -b(u_1,G(u_1),e)+b(u_2,G(u_2),e)= -b(e,e,G(u_1))-b(e,\delta,u_2).\label{est.t1}
\end{align}
Elementary algebra with \eqref{wformb}, the definition of $G(\bullet)$ and symmetry of $b(\bullet,\bullet,\bullet)$ with respect to the three variables show
\begin{align}
\trinl \delta \trinr^2=b(u_1,u_1,\delta)-b(u_2,u_2,\delta)=b(e,\delta,u_1)+b(e,\delta,u_2).\label{est.t2}
\end{align}
The combination of \eqref{est.t1}-\eqref{est.t2}, and Lemma \ref{Bound b} lead to
\begin{align}\label{boundedelta}
\frac{1}{2}\trinl e \trinr^2+\trinl \delta \trinr^2&\le b(e,\delta,u_1)-b(e,e,G(u_1))
\le C_{\rm S}\trinl e \trinr\trinl \delta \trinr\trinl u_1\trinr_{}+C_{\rm S}\trinl e \trinr^2\trinl G(u_1)\trinr_{}.
\end{align}
Suppose $0 <\gamma^2:= C^2_{\rm S}(\frac{1}{2}\trinl u\trinr^2+\trinl v\trinr^2) < (\sqrt{2}-1)^2$ and verify $\gamma < \frac{1}{\gamma} -2$. Hence there exist a real $\lambda$ with $\gamma < \lambda< \frac{1}{\gamma} -2$. The inequality \eqref{boundedelta} and a $\lambda$-weighted Young inequality lead to 
\begin{align}\label{est.uniq}
\frac{1}{2}\trinl e \trinr^2+\trinl \delta \trinr^2&\le \gamma (\sqrt{2}\trinl e \trinr \trinl \delta \trinr+ \trinl e \trinr^2 ) \le \gamma ((1+ \frac{\lambda}{2})\trinl e \trinr^2 + \lambda^{-1}  \trinl \delta \trinr^2 ).
\end{align}
This is equivalent to $(\frac{1}{2} - \gamma (1+\frac{\lambda}{2}))   \trinl e \trinr^2 + (1- \lambda^{-1} \gamma) \trinl \delta \trinr^2 \le 0$. Since each of the two previous factors in the lower bound are positive, this proves $\trinl e \trinr=0=\trinl \delta \trinr$ and it concludes the proof of uniqueness. 
\end{proof}

\begin{rem}[{\em a priori} and {\em a posteriori} criteria for uniqueness]\label{smallnessassumption}
	The {\em a priori} smallness assumption on data $C_{\rm S}M(f,\chi)<\sqrt{2}-1$ implies $C_{\rm S} N(f,u)<\sqrt{2}-1$ and so global uniqueness of the solution to \eqref{wform}. The first condition is {\em a priori}, but given the constant $\sqrt{2}-1$, $f$, and $\chi,$ $M(f,\chi)$ is hard to quantify. The second condition $C_{\rm S} N(f,u)<\sqrt{2}-1$ is {\em a posteriori} in the sense that $u$ can be replaced by some $\varphi\in K$. Once an approximation $u_\M$ of $u$ is known, some $\varphi\in K$  can be postprocessed by $u_\M$ (similar to the construction in \cite[Lemmas 3.3, 3.4]{brenner2012finite}) and then $N(f,u)$ can be bounded from above by $N(f,\varphi)$. If computed upper bound is small than $\sqrt{2}-1,$ this implies uniqueness of $(u,v).$
\end{rem}

\maketitle
\section{Regularity}
The regularity result in \cite{frehse1971differenzierbarkeitsproblem} will be employed for modified obstacles in the biharmonic obstacle problem. Given any obstacle $\widetilde{\chi}\in H^2(\Omega)\cap H^3_{\rm loc}(\Omega)$ with $\max \widetilde{\chi}(\partial \Omega)<0$, define a corresponding non-empty, closed and convex subset $K(\widetilde{\chi}):=\{\varphi\in H^2_0(\Omega): \varphi\ge \widetilde{\chi}\text{ a.e. }\text{in }\Omega \}$ of $H^2_0(\O)$ and notice $K=K(\chi)$ for the original obstacle $\chi$ from \eqref{wform}. Given any such $\widetilde{\chi},$ and $f\in L^2(\O),$ consider the problem that seeks the solution $\phi\in K(\widetilde{\chi})$ to
\begin{align}
a(\phi,\phi-\psi)&\le (f,\phi-\psi)_{L^2(\O)} \quad\text{  for all  }\psi\in K(\widetilde{\chi}). \label{VI}
\end{align}
\begin{thm}[Frehse 1971]\label{thm.frehse}
Let $\Omega$ be an open bounded connected subset of $\R^2$. If $\phi\in K(\widetilde{\chi})$ solves \eqref{VI} for $\widetilde{\chi}\in H^2(\Omega)\cap H^3_{\rm loc}(\Omega)$ with $\max \widetilde{\chi}(\partial \Omega)<0$, then $\phi\in H^2_0(\Omega)\cap H^3_{\rm loc}(\Omega)$.
\end{thm}
\begin{proof}
Frehse's result \cite[Theorem 1]{frehse1971differenzierbarkeitsproblem} shows $\phi\in H^2_0(\Omega)\cap H^3_{\rm loc}(\Omega)$ even under the much more involved assumption $\widetilde{\chi}\in H^3(\O)$ and $\max \widetilde{\chi}(\partial\Omega)\le 0$. The theorem at hand assures that $\max \widetilde{\chi}(\partial\Omega)<0$ and the proof will establish that Frehse's result can be adapted. The remaining parts of this proof establish that for an appropriate $\widehat{\chi}\in H^3(\O)$ constructed in the sequel, $\phi$ satisfies \eqref{VI} with an obstacle $\widehat{\chi}$. Since $\max \widetilde{\chi}(\partial\Omega)<0$ and $\phi\in H^2_0(\Omega)$, there exist $\epsilon>0$ and $\delta<0$ such that $\widetilde{\chi}<\delta<\phi$ in $\overline{N(2\epsilon,\partial\Omega)},$ where $N(2\epsilon,\partial\Omega):= \{x\in \Omega: {\rm{dist}}(x,\partial\Omega)<2\epsilon\}$. Select cut-off functions $0\leq\psi_1,\psi_2\in C^{\infty}(\overline{\Omega})$ such that $\psi_1+\psi_2\equiv 1 \text{ in } \overline{\Omega}$ and
	\begin{align}
	\psi_1=
	\begin{cases}
	1\quad \text{ in }\,\overline{N(\epsilon,\partial\Omega)},\\
	0\quad \text{ in }\,\Omega\setminus N(2\epsilon,\partial\Omega)\\
	\end{cases}
	\mbox{and}
	\hspace{5mm}
	\psi_2=
	\begin{cases}
	0\quad \text{ in }\,\overline{N(\epsilon,\partial\Omega)},\\
	1\quad \text{ in }\,\Omega\setminus N(2\epsilon,\partial\Omega).\\
	\end{cases}
	\end{align}
	Consider $\widehat{\chi}:=\delta\psi_1+\widetilde{\chi}\psi_2$ and derive the following three inequalities
	$$\widetilde{\chi}(x)<\delta=\widehat{\chi}(x)<\phi(x)\,\text{ for all } x\in \overline{N(\epsilon,\partial\Omega)},$$
	$$\widetilde{\chi}(x)=\widehat{\chi}(x)\le \phi(x)\,\text{ for all } x\in \Omega\setminus N(2\epsilon,\partial\Omega),$$
	$$\widetilde{\chi}(x)<\delta\psi_1(x)+\widetilde{\chi}(x)\psi_2(x)=\widehat{\chi}(x)<\delta<\phi(x)\,\text{ for all } x\in \big(\Omega\setminus\overline{N(\epsilon,\partial\Omega)}\big)\cap N(2\epsilon,\partial\Omega) .$$
	
	\noindent The above three inequalities imply $\widetilde{\chi}\le\widehat{\chi}\le \phi$ in $\overline{\Omega}$ and $\widetilde{\chi}\in H^3\big(\Omega\setminus N(\epsilon,\partial\Omega)\big).$ By construction, $\widehat{\chi}$ is the combination of a $H^3(\Omega)$ and a $C^{\infty}(\overline{\Omega})$ function, and hence, $\widehat{\chi}\in H^3(\Omega)$. Given $\widehat{\chi}$ as an obstacle, the solution $\phi \in K(\widehat{\chi})$ to \eqref{VI} also satisfies 
	\begin{align}
	a(\phi,\phi-\psi)&\le (f,\phi-\psi)_{L^2(\Omega)} \quad\text{  for all  }\psi\in K(\widehat{\chi}). \label{VItilde}
	\end{align}
	Since the obstacle $\widehat{\chi}$ of the problem \eqref{VItilde} belongs to $H^2_0(\Omega)\cap H^3(\Omega)$, \cite[Theorem 1]{frehse1971differenzierbarkeitsproblem} proves $\phi\in H^3_{\rm loc}(\Omega)$. 
\end{proof}
\noindent

\noindent
The final regularity result of the \vk\, obstacle problem relies on the following three lemmas.

\begin{lem}[{\cite[Equation (2.6)]{blum1990mixed}}, {\cite[Theorem 2]{blum1980boundary}}]\label{blum1980boundary}
	Let $\O$ be a bounded polygonal domain in $\R^2$. If $w\in H^2_0(\O)$ solves the biharmonic problem, $a(w,\varphi)=f(\varphi)$ for all $\varphi\in H^2_0(\O)$, with data $f\in H^{-1}(\O)$ $($resp. $L^2(\O))$, then $w\in H^3_{\rm loc}(\O)\cap H^{2+\alpha}(\O)$ $($resp. $H^4_{\rm loc}(\O)).$ If the bounded Lipschitz domain $\O$ has a $C^{2+\gamma}$ boundary for some $0<\gamma<1$ and $f\in L^2(\O)$ $($resp. $H^{-1}(\O))$, then the solution $w$ belongs to $H^4(\O)$ $($resp. $H^3(\O))$.
\end{lem}
\begin{lem}[{\cite[Theorem 7]{blum1980boundary}}]\label{blum1980boundaryvk}
Let $\O$ be a bounded polygonal domain in $\R^2$. If $(w_1,w_2)\in H^2_0(\O)\times H^2_0(\O)$ is a solution to the \vket, $a(w_1,\varphi_1)+2b(w_1,w_2,\varphi_1)+a(w_2,\varphi_2)-b(w_1,w_1,\varphi_2)=f(\varphi_1)$ for all $(\varphi_1,\varphi_2)\in H^2_0(\O)\times H^2_0(\O)$, with data $f\in H^{-1}(\O)$, then $(w_1,w_2)\in H^{2+\alpha}(\O)\times H^{2+\alpha}(\O)$. 
\end{lem}

\noindent The remaining parts of this section return to \eqref{wform} with $f \in L^2(\Omega)$ and a polygonal domain $\Omega$.

\begin{lem}\label{[u,v]belongsH-1}
	If $(u,v)\in K\times H^2_0(\O)$ solves \eqref{wform}, then $[u,v]\in H^{-1}(\Omega)$.
\end{lem}
\begin{proof}
	The Sobolev embedding $H^{1+\epsilon}(\O)\hookrightarrow L^{\infty}(\O)$ and $u\in H^2_0(\O)$ imply $[u,u]\in H^{-1-\epsilon}(\Omega)$ for any $\epsilon>0.$ A shift theorem \cite[Theorem 8]{bacuta2002shift} in \eqref{wformb} shows $v\in H^{2+\alpha-\epsilon}(\O)$ for $1/2<\alpha\le 1.$ 
	Given $\alpha,$ choose $\epsilon$ such that $\alpha-\epsilon>1/2$. 
	Then, \cite[Lemma 2.2]{brenner2017c} 
	implies \[([u,v],\varphi)_{L^2(\O)}= -\int_{\O} \mathrm{cof}(D^2v)\nabla u\cdot\nabla \varphi \,dx\le \|\mathrm{cof}(D^2v)\|_{L^4(\O)}\|\nabla u\|_{L^4(\O)}\|\nabla \varphi\|_{L^2(\Omega)}\] for all $\varphi\in H^1_0(\Omega).$ This and the Sobolev embeddings $H^2(\O)\hookrightarrow W^{1,4}(\O)$, $H^{2+\alpha-\epsilon}(\O)\hookrightarrow W^{2,4}(\O)$ conclude the proof.
\end{proof}

\begin{thm}[\rm Regularity for \vk\, obstacle problem]\label{uH2+delta}
	Let $\O$ be a bounded polygonal domain in $\R^2$. If $(u,v)\in K\times H^2_0(\O)$ solves \eqref{wform}, then $u,v\in H^{2+\alpha}(\Omega)\cap H^3_{\rm loc}(\Omega)\cap C^2(\O).$
\end{thm}
\begin{proof}
	Let $(u,v)$ solve \eqref{wform} and let $w\in K(\chi)$ solve
	\begin{align}
	a(w,w-\varphi)&\le (f+[u,v],w-\varphi)_{L^2(\Omega)} \quad\text{  for all  }\varphi\in K(\chi). \label{regvke1}
	\end{align}
	Let $w_1\in H^2_0(\Omega)$ be the Riesz representation of $-[u,v]$ in the Hilbert space $\big(H^2_0(\O),a(\bullet,\bullet)\big)$, i.e., $w_1$ satisfies $a(w_1,\varphi_1)= -([u,v],\varphi_1)_{L^2(\O)}$ for all $\varphi_1\in H^2_0(\Omega).$ Lemmas \ref{[u,v]belongsH-1} and \ref{blum1980boundary} show that $w_1\in H^3_{\rm loc}(\Omega)\cap H^{2+\alpha}(\Omega)$. Translate the obstacle $\chi$ of \eqref{regvke1} to $\chi+w_1$ and set $\widetilde{w}:=w+w_1$ to obtain
	\begin{align}
	a(\widetilde{w},\widetilde{w}-\widetilde{\varphi})&\le (f,\widetilde{w}-\widetilde{\varphi})_{L^2(\O)} \quad\fl\widetilde{\varphi}\in K(\chi+w_1) \label{transregvke1}.
	\end{align}
	Since the obstacle $\chi+w_1\in H^3_{\rm loc}(\Omega)$, Theorem \ref{thm.frehse} implies $\widetilde{w}\in H^3_{\rm loc}(\Omega)$. Also $w_1\in H^3_{\rm loc}(\Omega)$ implies $w=\widetilde{w}-w_1\in H^3_{\rm loc}(\Omega)$. The solution $u$ to \eqref{wform} also solves \eqref{regvke1}. The uniqueness of the solution in \eqref{VI} implies that $u\in H^3_{\rm loc}(\Omega)$.\\
	Let the contact region $\mathscr{C}:=\{x\in\Omega:u(x)=\chi(x)\}.$ Define a cut-off function $\xi \in C^\infty(\overline \Omega)$ with $\xi\equiv 1$ in $ N(\epsilon,\partial\Omega)$ for some $\epsilon>0$ such that $\overline{ N(2\epsilon,\partial\Omega)}\cap\mathscr{C}=\emptyset$, i.e., ${\rm supp}(\xi)\subset  N(2\epsilon,\partial\Omega)$ keeps a positive distance to $\mathscr{C}$. The strong form of \eqref{wformb} and elementary manipulations show
	\begin{equation}
	\Delta^2 v=-{{\frac{1}{2}}}[u,u]=-\frac{1}{2}[\xi u,\xi u]-[(1-\xi)u,\xi u]-\frac{1}{2}[(1-\xi)u,(1-\xi)u].
	\end{equation}
	Let $v_1\in H^2_0(\Omega)$ solve $\Delta^2v_1=f_1$ for $f_1:=-[(1-\xi)u,\xi u]-\frac{1}{2}[(1-\xi)u,(1-\xi)u]$. Since $u\in H^2_0(\Omega)\cap H^3_{\rm loc}(\Omega),$ $\xi u\in H_0^2(\Omega)$, $(1-\xi)u\in H^3(\Omega)$ and $f_1\in H^{-1}(\O)$. Lemma \ref{blum1980boundary} leads to $v_1\in H^{2+\alpha}(\Omega)$. Also, $\Delta^2(v-v_1)=-{\small\frac{1}{2}}[\xi u,\xi u]$ in $\Omega$. Since ${\rm supp}(\xi)\cap\mathscr{C}=\emptyset,$ \eqref{wforma} implies  $\Delta^2 u=f+[u,v]$ in ${\rm supp}(\xi)$. Since $(1-\xi) u\in H^3(\Omega)$ and $v_1\in H^{2+\alpha}(\Omega),$ it follows from the arguments in Lemma \ref{[u,v]belongsH-1} that $f_2:=f+[\xi u,v_1]+[(1-\xi) u,v]-\Delta^2((1-\xi) u)\in H^{-1}(\Omega)$. This and elementary manipulations lead to
	\begin{align*}\nonumber
	\Delta^2(\xi u)&=\Delta^2 u-\Delta^2((1-\xi) u)
	=f+[\xi u,v]+[(1-\xi) u,v]-\Delta^2((1-\xi) u)\\\nonumber
	&=f+[\xi u,v-v_1]+[\xi u,v_1]+[(1-\xi) u,v]-\Delta^2((1-\xi) u)\\
	&=f_2+[\xi u,v-v_1].\nonumber
	\end{align*}
	In other words, $(\xi u,v-v_1)$ solves the von K\'arm\'an equations for the right-hand side $f_2 \in H^{-1}(\Omega)$ and $\xi u\in H^{2+\alpha}(\Omega)$. 
	Since $\xi u, (1-\xi)u\in H^{2+\alpha}(\Omega),$ it follows $u=\xi u+(1-\xi)u\in H^{2+\alpha}(\Omega)$. Return to the proof of Lemma \ref{[u,v]belongsH-1} with the improved regularity $u\in H^{2+\alpha}(\O)$ to deduce that $[u,u]\in H^{-1}(\O)$. Since $v=G(u)$ solves \eqref{wformb}, this shows $v\in H^{2+\alpha}(\O)\cap H^3_{\rm loc}(\Omega)$.
	\smallskip
	\noindent\\
	The above arguments imply $u,v\in H^{2+\alpha}(\Omega)$, for $\alpha>1/2$, and the Sobolev embedding $ H^{2+\alpha}(\Omega)\hookrightarrow W^{2,4}(\Omega)$ shows $[u,u],[u,v]\in L^2(\Omega)$. By Lemma \ref{blum1980boundary}, the solution to $\Delta^2 v=-\frac{1}{2}[u,u]$ belongs to $H^4_{\rm loc}(\Omega)$. Then, the continuous Sobolev embedding $H^4_{\rm loc}(\Omega)\hookrightarrow C^2(\O)$ implies $v\in C^2(\O)$. Since $u\in H^2_0(\O)$, $\chi\in C^2(\O)$, and $\max \chi(\partial \Omega)<0$, the arguments in the proof of Theorem \ref{thm.frehse} lead to $\widetilde{\chi}\in C^2(\overline{\O})$ such that $\chi\le\widetilde{\chi}\le u$. This shows that $u\in K(\widetilde{\chi})$, and hence with $\widetilde{f}:=f+[u,v]\in L^2(\O)$, $u$ solves
	\begin{align}
	a(u,u-\varphi)&\le (\widetilde{f},u-\varphi)_{L^2(\O)} \quad\fl\varphi\in K(\widetilde{\chi}). \label{tildechi}
	\end{align}
	\cite[Appendix A]{brenner2012finite} establishes the regularity result for the biharmonic obstacle problem \eqref{tildechi}, which implies that the solution $u$ belongs to $C^2(\O)$. This concludes the proof.
\end{proof}
\noindent

\begin{rem}[$C^2(\overline{\Omega})$ \rm regularity]\label{C2regularity}
	If the bounded Lipschitz domain $\O$ has a $C^{2+\gamma}$ boundary for some $0<\gamma<1,$ then any solution $(u,v)$ to \eqref{wform} belongs to $C^2(\overline{\Omega})\times C^2(\overline{\Omega})$. In fact, $[u,u]\in L^2(\O)$, Lemma \ref{blum1980boundary}, and continuous Sobolev embedding $H^4(\O)\hookrightarrow C^2(\overline{\O})$ imply that the solution $v$ to \eqref{wformb} belongs to $C^2(\overline{\O})$. An application of Lemma \ref{blum1980boundary} to the  arguments of \cite[Appendix A]{brenner2012finite} for \eqref{tildechi} conclude that the solution $u$ to \eqref{wforma} belongs to $C^2(\overline{\O})$.%
\end{rem}

\section{Morley finite element approximation}
The first subsection discusses some preliminaries on the Morley FEM and interpolation and enrichment operators. The second subsection derives the existence, uniqueness under a computable smallness assumption and an {\em a priori} bound of the discrete solution.

\subsection{Preliminaries}\label{sec.preliminaries}
\noindent Let $\T$ be an admissible and regular triangulation of the polygonal bounded Lipschitz domain $\Omega$ into triangles in $\mathbb R^2$, let $h_T$ be the diameter of a triangle $T \in \T$ and $h_{\max} :=\max_{T \in \T} h_T$. For any $\epsilon>0$, let $\mathbb{T}(\epsilon)$ denote the set of all triangulations $\cT$ with $h_{\max} <\epsilon$. For a non-negative integer $m$, let ${\mathcal P}_{m}(\mathcal{T})$ denote the space of piecewise polynomials of degree at most $m$. Let $\Pi_0$ denote the $L^2$ projection onto the space ${\mathcal P}_0(\mathcal{T})$ of piecewise constants and let $\mathcal{E}$ and $\mathcal{V}$ be the set of edges and vertices of $\T$, respectively. The set of all internal edges (resp. boundary edges) of $\cE $ is denoted by $\cE (\Omega)$ (resp. $\cE (\partial\Omega)$). Denote the set of internal vertices (resp. vertices on the boundary) of $\cT $ by $ \mathcal V (\Omega)$ (resp.  $ \mathcal V(\partial\Omega)$). The nonconforming Morley finite element space $\M(\T)$  is defined by
\[
\M(\T)=\left \{ \varphi_\M \in {\mathcal P}_2(\T){{\Bigg |}}
\begin{aligned}
& \varphi_\M \text{ is continuous at } \mathcal V (\Omega) \text{ and vanishes at } \mathcal V(\partial\Omega) \\
& \forall E\in \cE(\Omega),\; \int_{E}\left[\frac{\partial \varphi_{\rm M}}{\partial n}\right]_E\ds=0;\;\forall E\in \cE (\partial\Omega),\; \int_{E}\frac{\partial \varphi_{\rm M}}{\partial n}\ds=0
\end{aligned}
\right\}
\] 
where $n$ denotes the unit outward normal to the boundary $\partial \O$ of $\O$ and $\left[\varphi_{\rm M}\right]_E$ is the jump of $\varphi_{\rm M}$ across any interior edge $E$. Let the Morley element space $\M(\T)$ be equipped with the piecewise energy norm $\trinl\bullet\trinr_{\pw}$ defined by $ \trinl\varphi_{\rm M}\trinr_{\pw}^2:=\sum_{T\in \mathcal{T}}\|D_{\rm pw}^2\varphi_{\rm M}\|_{L^2(T)}^2$ for any $\varphi_{\rm M}\in \M(\T)$, where for $j=0,1,2;$ let $D_{\rm pw}^j$ be defined as $D_{\rm pw}^0\varphi_{\rm M}=\varphi_{\rm M}$, $D_{\rm pw}^1\varphi_{\rm M}=\nabla_{\rm pw}\varphi_{\rm M}$, and $D_{\rm pw}^2\bullet$ is the piecewise Hessian. Given the obstacle $\chi\in H^2(\O)$ with $\max\chi(\partial\Omega)<0$, define the discrete analogue \cite{brenner2013morley}
$$K(\chi,\T):=\Big\{\varphi_\M \in \M(\T)\,\big| \,\,\chi(p) \le \varphi_\M(p) \fl p \in \mathcal V\Big\}$$
to $K$. The Morley nonconforming FEM for \eqref{wform} seeks $(u_\M,v_\M) \in K(\chi,\T)\times \M(\T)$ such that
\begin{subequations}\label{Morleywform}
	\begin{align}
	& a_\pw(u_\M,u_\M-\varphi_1)+ 2b_\pw(u_\M,u_\M-\varphi_1, v_\M)\le (f,u_\M-\varphi_1)_{L^2(\O)}   \; \; \fl \varphi_1\in K(\chi,\T)\label{Morleywforma},\\
	& a_\pw(v_\M,\varphi_2)-b_\pw(u_\M,u_\M,\varphi_2)   =0       \; \;     \fl \varphi_2 \in \M(\T)\label{Morleywformb}.
	\end{align}
\end{subequations}
Here and throughout this paper, for all $ \eta_\M,w_\M, \varphi_\M \in \M(\T)$, define 
\begin{align}
&a_\pw(\eta_\M,\varphi_\M):=\int_{\O} D_{\rm pw}^2 \eta_\M:D_{\rm pw}^2\varphi_\M {\,dx}, \nonumber\\
& b_\pw(\eta_\M,w_\M,\varphi_\M):=-\half\int_{\O} [\eta_\M,w_\M]_{\rm pw}\varphi_\M \,dx.\label{defn-apw}
\end{align}
Note that $b_\pw(\bullet,\bullet,\bullet)$ is symmetric with respect to the first two arguments.

\begin{lem}[\rm Morley interpolation \cite{carstensen2014discrete,gallistl2014morley, carstensen2018prove}] \label{Morley_Interpolation} 
The Morley interpolation $I_{\rm M}: H^2_0(\O) \rightarrow {\rm M}(\T)$ is defined, for $\varphi\in H^2_0(\O)$, by (the degrees of freedom for the Morley finite element)
\begin{align*}
& (I_{\rm M} \varphi)(z)=\varphi(z) \text{ for any } z\in\mathcal{V} \text{ and }
\int_E\frac{\partial I_{\rm M} \varphi}{\partial n_E}\ds=\int_E\frac{\partial \varphi}{\partial n_E}\ds \text{ for any  edge } E\in\mathcal{E},
\end{align*}
and satisfies $(a)$-$(c)$ for all $\psi\in H^2(T)$, $T\in\T$, and all $\varphi \in H^2_0(\O) \cap H^{2+\alpha}(\Omega)$.
\begin{itemize}
	\item[$(a)$] (integral mean property of the Hessian) $D^2_{\rm pw} I_{\rm M} =\Pi_0 D_{\rm pw}^2$ in $H^2_0(\O)$,
	\item[$(b)$] (approximation and stability)
	$$\|h_T^{-2}(1-I_{\rm M})\psi\|_{L^2(T)}+\|h_T^{-1}D_{\rm pw}(1-I_{\rm M})\psi\|_{L^2(T)} \lesssim
	\|D_{\rm pw}^2 (1-I_{\rm M})\psi\|_{L^2(T)},$$
	\item[$(c)$] $\|D_{\rm pw}^2 (1-I_{\rm M})\varphi\|_{L^2(\Omega)}\lesssim h_{\max}^{\alpha} \| \varphi\|_{{2+\alpha}}$.
	
\end{itemize}  
\end{lem}

\begin{lem}[\rm Enrichment/Conforming Companion \label{hctenrich} \cite{gallistl2014morley, carstensen2018prove}] There exists a linear operator $E_{\rm M}:{\rm M}(\T)\to H^2_0(\O)$ such that any $\varphi_{\rm M} \in {\rm M}(\T)$ satisfies $(a)$-$(d)$ with a universal constant $\Lambda$ that depends on the shape-regularity of $\T$ but not on the mesh-size $h_\T\in {\mathcal P}_2(\T)$.\\
	$(a)$ $I_{\rm M} E_{\rm M} \varphi_{\rm M}= \varphi_{\rm M}$, \qquad $(b)$ $\Pi_0(\varphi_{\rm M} - E_{\rm M}  \varphi_{\rm M}) =0$,\qquad $(c)$ $\Pi_0D_{\rm pw}^2(\varphi_{\rm M} - E_{\rm M}  \varphi_{\rm M}) =0$, \\
	$(d)$ $\sum_{j=0}^2\| h_\T^{j-2}D_{\rm pw}^j(\varphi_{\rm M}-E_{\rm M}\varphi_{\rm M})\|_{L^2(\Omega)} \le \Lambda \min_{\varphi \in H^2_0(\O)}\| D_{\rm pw}^2(\varphi_{\rm M}-\varphi)\|_{L^2(\Omega)}.$
\end{lem}

\begin{rem}\label{remInterEnt}
Lemmas \ref{Morley_Interpolation} and \ref{hctenrich} lead for all $\varphi_\M,\,w_\M, \,\psi_\M \in \M(\T)$ and $\psi \in H^2_0(\O)$ to 
$$a_{\rm{pw}}(\varphi_{\rm M}, E_{\rm{M}} \psi_{\rm{M}}-\psi_{\rm{M}})=a_{\rm{pw}}(\varphi_{\rm M}, E_{\rm{M}}I_{\rm{M}} \psi-\psi)=b_{\rm{pw}}(\varphi_{\rm M},w_{\rm M}, E_{\rm{M}} \psi_{\rm{M}}-\psi_{\rm{M}})=0.$$
\end{rem}
\begin{lem} [\rm Bounds for $a_{\rm{pw}}(\bullet,\bullet)$ {\cite[Lemmas 4.2, 4.3]{brenner2013morley}}]\label{ah.bound} Any $\varphi \in H^{2+\alpha}(\Omega)$, $\psi \in H^2_0(\O)\cap H^{2+\alpha}(\Omega)$, $\psi_{\rm M}, \varphi_{\rm M} \in {\rm M}(\T)$ satisfy $(a)$-$(c).$
\begin{itemize}
	\item[$(a)$] $a_{\rm{pw}}(\varphi,E_{\rm M}\psi_{\rm M}-\psi_{\rm M})\lesssim h_{\max}^\alpha\|\varphi \|_{2+\alpha}\| \psi_{\rm M} \|_{\rm pw},$
	\item[$(b)$]$a_{\rm{pw}}(\varphi,I_{\rm M} \psi-\psi)\lesssim h_{\max}^{2\alpha}\|\varphi \|_{2+\alpha}\| \psi \|_{2+\alpha},$
	\item[$(c)$] 
	the scalar product $a_{\rm pw}(\bullet,\bullet)$ is elliptic in the sense that
	$a_{\rm{pw}}(\varphi_{\rm{M}},\varphi_{\rm{M}}) =  \trinl\varphi_{\rm{M}}\trinr^{2}_{\rm pw}$.
\end{itemize}
\end{lem}

\noindent Recall from \eqref{ctsembedding}, the Sobolev (resp. Friedrichs) constant $C_{\rm S}$ in the Sobolev embedding $H^2_0(\O)\hookrightarrow C(\overline{\Omega})$ (resp. {$C_{\rm F}$ in  $H^2_0(\O)\hookrightarrow L^2(\O)$}). Recall the index $\alpha'$ of elliptic regularity. 
	
\begin{thm}[discrete Sobolev and Friedrichs inequalities]\label{thm.discretesobolev} For $0<\alpha'<1$, set $\beta=\alpha'$ and for $1\le\alpha'$ and any $0<\epsilon<1$, set $\beta=1-\epsilon$. Then there exist positive constants $C(\beta)$, and $C(\alpha')$ such that $C_{\rm dS}:=C_{\rm S}+C(\beta)h_{\max}^{\beta}$ and $C_{\rm dF}:=C_{\rm F}+C(\alpha')h_{\max}^{\alpha'}$ satisfy for any $v+v_{\rm M} \in H^2_0(\O)+{\rm M}(\T)$
	$$(a)\,\| v+v_{\rm M} \|_{L^\infty(\O)} \le C_{\rm dS}\trinl v+v_{\rm M} \trinr_{\rm pw}\quad\text{and}\quad(b)\,\| v+v_{\rm M} \|_{L^2(\O)} \le C_{\rm dF}\trinl v+v_{\rm M} \trinr_{\rm pw}.$$
\end{thm}
\begin{proof}
	The point of the theorem is to get sharp estimates of $C_{\rm dS}$ and $C_{\rm dF}$, otherwise this result is a direct consequence of e.g. \cite[Lemma 4.7]{carstensen2017nonconforming}.
	\medskip
	\noindent\\
	$(a)$ The piecewise uniformly continuous function $ v+v_{\rm M} $ has a maximum norm that is the supremum of all integrals $\int_{\O}  (v+v_{\rm M}) \varphi \,dx$ for $\varphi \in L^1(\O)$ with $\| \varphi \|_{L^1(\O)} = 1.$ Given $\varphi \in L^1(\O)$ with $\| \varphi \|_{L^1(\O)}= 1$, let $z \in H^2_0(\O)$ solve \begin{equation}\label{weak.ds}
	a(z,\bullet)=\langle\varphi,\bullet\rangle_{L^1(\O)},
	\end{equation}
	where the duality $\langle\bullet,\bullet\rangle_{L^1(\O)}$ extends the $L^2$ scalar product. For any $0<\epsilon<1,$ the embedding $H^{1+\epsilon}(\O) \hookrightarrow L^\infty(\O)$ is continuous. This implies $\langle\varphi,\bullet\rangle_{L^1(\O)} \in H^{-(1+\epsilon)}(\O)$. For $0<\alpha'<1$, choose $0<\epsilon<1$ such that $0<\alpha'<1-\epsilon$, set $\beta=\alpha'.$ For $1\le\alpha'$ and any $0<\epsilon<1$, set $\beta=1-\epsilon$. The shift theorem \cite[Theorem 8]{bacuta2002shift} in elliptic regularity shows $z \in H^{2+\beta}(\O)$. With bound $C(\beta,\O)$ of the embedding $H^{1+\epsilon}(\O) \hookrightarrow L^\infty(\O)$ and since $\| \varphi \|_{L^1(\O)} = 1$,
	\begin{align}\label{sup.ds}
	\| z\|_{2+\beta} \lesssim \| \langle\varphi,\bullet\rangle_{L^1(\O)} \|_{H^{-(1+\epsilon)}}=\sup_{0 \neq\psi \in H^{1+\epsilon}_0(\O)}\frac{(\varphi,\psi)_{L^2(\O)}}{\| \psi \|_{L^{\infty}(\O)}}\frac{\| \psi \|_{L^{\infty}(\O)}}{\| \psi \|_{1+\epsilon}} \le C(\beta,\O).
	\end{align}
	Given $v \in H^2_0(\O)$ and $v_{\rm M} \in \M(\cT)$, let $w \in H^2_0(\O)$ solve $a_\pw(w,\bullet)=a_\pw(v+v_{\rm M},\bullet) \in H^{-2}(\O).$ Set $\delta:=w-v-v_{\rm M}$ and recall $a_\pw(\delta,z)=0.$ This, $\| \varphi \|_{L^1(\O)} = 1$, and the Sobolev constant $C_{\rm S}$ lead to
	\begin{align*}
	(v+v_{\rm M}, \varphi)_{{L^2(\O)}} =(w,&\varphi)_{L^2(\O)} -(\delta,\varphi)_{L^2(\O)} \le \|w\|_{L^\infty(\O)}-(\delta,\varphi)_{L^2(\O)}\nonumber\\
	& \le C_{\rm S} \trinl w \trinr-(\delta,\varphi)_{L^2(\O)}\le C_{\rm S} \trinl v+v_{\rm M} \trinr_\pw-(\delta,\varphi)_{L^2(\O)}.
	\end{align*}
	Since $w-v-E_{\rm M} v_{\rm M} \in H^2_0(\O)$, \eqref{weak.ds}, H\"older inequality, $a_\pw(\delta,z)=0$, Lemma \ref{hctenrich}.$c$-$d$, inverse estimate, Lemma \ref{Morley_Interpolation}.$c$, and \eqref{sup.ds} read
	\begin{align*}
	(\delta,\varphi)_{L^2(\O)}
	&\le a(z,w-v-E_{\rm M} v_{\rm M} ) + \| \varphi \|_{L^1(\O)} \|E_{\rm M}v_{\rm M}-v_{\rm M} \|_{L^\infty(\O)}\\
	&\le a_\pw(z-I_\M z,v_{\rm M}-E_{\rm M}v_{\rm M})+  C_{\rm inv}\|h_\cT^{-1} (v_{\rm M}-E_{\rm M}v_{\rm M}) \|_{L^2(\O)}\\
	&\lesssim \trinl v+v_{\rm M} \trinr_\pw (h_{\max}^{\beta}\| z\|_{2+\beta}+ C_{\rm inv}\Lambda h_{\max})\lesssim h_{\max}^{\beta}\trinl v+v_{\rm M} \trinr_\pw.
	\end{align*}
	The combination of the last and second-last displayed inequalities conclude the proof of $(a)$ with the constant $C_{\rm S}+C(\beta)h_{\max}^{\beta}$.
	
	\smallskip
	\noindent
	$(b)$ Given any $\varphi \in L^2(\O)$ with $\| \varphi \|_{L^2(\O)}= 1$, let $z \in H^2_0(\O)$ solve $a(z,\bullet)=(\varphi,\bullet)_{L^2(\O)} \in L^{2}(\O).$ Then, $z \in H^{2+\alpha'}(\O)$ \cite[Theorem 2]{blum1980boundary}. Note that
	\begin{align*}
	\|v+v_{\rm M}\|_{L^2(\O)} =\sup_{\substack{\varphi \in L^2(\O)\\\| \varphi\|_{L^2(\O)}=1}}{(v+v_{\rm M},\varphi)_{L^2(\O)}}.
	\end{align*}
	Analogous arguments of part $(a)$ apply with $\|w\|_{L^2(\O)} \le C_{\rm F} \trinl w\trinr$ and replace $L^{\infty}(\O)$ and $\beta$ by $L^{2}(\O)$ and $\min\{\alpha',4\}$, respectively. This concludes the proof of $(b)$ with the constant $C_{\rm F}+C(\alpha')h_{\max}^{\alpha'}$.
\end{proof}

\begin{lem}[\rm Bounds for $b_{\rm{pw}}(\bullet,\bullet,\bullet)$ {\cite[Lemma 2.6]{ carstensen2019adaptive}}]\label{Bound bpw}
Any $\eta,\varphi,\phi \in H^2_0(\O)+{\rm M}(\T)$ satisfy\\\\          \medskip\qquad\qquad\qquad\qquad\qquad $b_{\rm{pw}}(\eta,\varphi,\phi) \le C_{\rm dS} \trinl\eta \trinr_{\rm{pw}}\trinl\varphi\trinr_{\rm{pw}}\trinl\phi\trinr_{\rm{pw}}.$
\end{lem}

\subsection{Existence, uniqueness, and {\em a priori} bound of the discrete solution}
This section establishes the well-posedness of the discrete problem \eqref{Morleywform}. The discrete analogue $G_\M: \M(\T) \rightarrow \M(\T)$ of $G$ in \eqref{defnG} is characterised by $G_\M(\phi_\M)\in \M(\T)$ and
$$ a_\pw(G_\M(\varphi_\M),\psi_\M)=b_\pw(\varphi_\M,\varphi_\M,\psi_\M)         \; \;     \fl \varphi_\M,\psi_\M\in \M(\T).$$
This gives rise to the energy functional $j_\pw(\bullet):K(\chi,\T) \rightarrow \R$ defined for all $\xi_\M \in K(\chi,\T)$ by
$$j_\pw(\xi_\M):=\half \trinl\xi_\M\trinr^2_{\rm pw}+\half \trinl G_\M(\xi_\M)\trinr^2_{\rm pw}-(f,\xi_\M)_{L^2(\O)}.$$

\begin{thm}[Existence, {\em a priori} and uniqueness condition]\label{apostexistctsuniqdisc}
Given $(f,\chi)\in L^2(\O)\times H^2(\O)$ with $\max \chi(\partial\Omega)<0$, there exists a minimizer $u_{\rm M} \in K_{\rm M}$ of $j_{\rm pw}(\bullet)$ in $K_{\rm M};$ each minimizer $u_{\rm M}$ and $v_{\rm M}:=G_{\rm M}(u_{\rm M})$ solves \eqref{Morleywform}. There is a positive constant $C_{\rm d}(\chi)$ that depends only on $\chi$ such that any solution $(u_{\rm M},v_{\rm M})$ to \eqref{wform} satisfies $(a)$-$(b).$\\
$(a)$ $\frac{1}{2}\trinl u_{\rm M}\trinr_{\rm pw}^2+\trinl v_{\rm M} \trinr_{\rm pw}^2
\le N_{\rm d}^2(f,u_{\rm M}):=2j_{\rm pw}(u_{\rm M})+2C^2_{\rm dF}\|f\|_{L^2(\O)}^2$

$\qquad\qquad\qquad\qquad\, \le M_{\rm d}^2(f,\chi):=C_{\rm d}(\chi)+3C^2_{\rm dF}\|f\|^2_{L^2(\O)}.$\\
$(b)$ If $C^2_{\rm dS}(\frac{1}{2}\trinl u_{\rm M}\trinr^2+\trinl v_{\rm M}\trinr^2)< (\sqrt{2}-1)^2,$ then $(u_{\rm M},v_{\rm M})$ is the only solution to \eqref{wform}. 
\end{thm}
\begin{proof}
	Analogous arguments as in Theorem \ref{th:existence:contslevel} show the existence of a minimizer $u_\M$ of $j_\pw(\bullet)$ in $K(\chi,\T)$ and $v_\M:=G_\M(u_\M)$ defines a solution $(u_\M,v_\M)$ to \eqref{Morleywform}. Since $u_\M$ is the global minimizer of $j_\pw(\bullet)$, the Young inequality, and a rearrangement of terms imply for any $\varphi_\M \in K(\chi,\T)$ that
	\begin{align*}
	&\trinl u_\M\trinr_\pw^2+\trinl v_\M\trinr_\pw^2
	\le 2j_\pw(\varphi_\M)+2(f,u_\M)_{L^2(\O)}\le 2j_\pw(\varphi_\M)+2C^2_{\rm dF}\|f\|_{L^2(\O)}^2+\frac{1}{2}\trinl u_\M\trinr^2.
	\end{align*}
 This implies
	\begin{align}\label{apostdiscaprioribound}
	&\frac{1}{2}\trinl u_\M\trinr_\pw^2+\trinl v_\M\trinr_\pw^2
	\le 2j_\pw(\varphi_\M)+2C^2_{\rm dF}\|f\|_{L^2(\O)}^2=:N_{\rm d}^2(f,\varphi_\M).
	\end{align}
	Given $\psi$ and $\varphi:=\chi\psi$ from the proof of Theorem \ref{thm.contsdep}, $\varphi_\M:=I_{\rm M}(\varphi)\in K(\chi,\T).$ The properties of Morley interpolation Lemma \ref{Morley_Interpolation}.$a$, definition of $G_\M(\bullet)$, the bounds of $b_\pw(\bullet,\bullet,\bullet)$ and Lemma \ref{Bound bpw} lead to
	\begin{align*}
	\trinl I_{\rm M}(\varphi)\trinr_\pw\le \trinl  \varphi\trinr\,\,\text{ and }\,\trinl G_\M(I_{\rm M}(\varphi))\trinr_\pw\le C_{\rm dS}\trinl \varphi\trinr^2.
	\end{align*} 
	The combination of above inequalities, the bound of $\trinl\varphi\trinr$ from \eqref{boundvarphi} conclude $(a)$ with $C_d(\chi):=2C^2(\psi)\|\chi\|^2_{H^2(\Omega_+)}+C_{\rm dS}^2C^4(\psi)\|\chi\|^4_{H^2(\Omega_+)}$ and $M_{\rm d}^2(f,\chi):=C_{\rm d}(\chi)+3C^2_{\rm dF}\|f\|^2_{L^2(\O)}$. The {\em a priori} bound in the equation \eqref{apostdiscaprioribound} and $u_{\rm M}$ being the minimizer of $j_\pw(\bullet)$ imply $ N_{\rm d}(f,u_{\rm M})\le N_{\rm d}(f,I_{\rm M}(\varphi))\le M_{\rm d}(f,\chi)$. 
	\medskip
	
	\noindent
	Let  $(u_\M^j,v_\M^j) \in K(\chi,\T)\times \M(\T)$ solve \eqref{Morleywform} for  $j=1,2$ and define $e:=u_\M^1-u_\M^2$ and $\delta:=v_\M^1-v_\M^2$. The test functions $u_\M^2$ (resp. $u_\M^1$) in \eqref{Morleywforma} and $\delta$ in \eqref{Morleywformb}, and a simplification lead to \eqref{est.t1} and \eqref{est.t2} with $\trinl\bullet\trinr,b(\bullet,\bullet,\bullet),u_1,u_2,v_1,v_2$ replaced by $\trinl\bullet\trinr_{\rm pw},b_{\rm pw}(\bullet,\bullet,\bullet),u_\M^1,u_\M^2,v_\M^1,v_\M^2.$ With this substitution, the algebra in \eqref{est.t1}-\eqref{est.uniq} holds verbatim with the further substitution of $C_{\rm S}$ and $M(f,\chi)$ by $C_{\rm dS}$ and $M_{\rm d}(f,\chi).$ Further details are omitted to conclude $e=0=\delta$ and this proves uniqueness.
\end{proof}
\begin{rem}[{\em a priori} and {\em a posteriori} criteria for discrete uniqueness]\label{discsmallnessassumption}
	The {\em a priori} smallness assumption on data $C_{\rm dS}M_{\rm d}(f,\chi)<\sqrt{2}-1$ implies the {\em a posteriori} smallness assumption $C_{\rm dS} N_{\rm d}(f,u_\M)<\sqrt{2}-1$ and so global uniqueness of the solution to \eqref{Morleywform}. 
\end{rem}

\section{{\em A priori} error analysis}
This section establishes an {\em a priori} error estimates of Morley FEM for the \vk\, obstacle problem with small data. 
\subsection{Main result}
Recall $M(f,\chi)$ and $M_{\rm d}(f,\chi)$ from Theorems \ref{thm.contsdep} and \ref{apostexistctsuniqdisc}, the Sobolev and Friedrichs (resp. its discrete versions) constants $C_{\rm S}$ and $C_{\rm F}$ (resp. $C_{\rm dS}$ and $C_{\rm dF}$) from \eqref{ctsembedding} and Theorem \ref{thm.discretesobolev}, and $\beta$, $\alpha$, $C(\beta)$, and $C(\alpha)$ from Theorem \ref{thm.discretesobolev}. The following theorem establishes for small data an {\em a priori} energy norm error estimates that is quasi-optimal plus linear convergence.
\begin{thm}[\rm{Energy norm estimates}] \label{thm.err}
	For a given $f\in L^2(\O),\chi\in C^2(\O)$ with $\max \chi(\partial\Omega)<0$ and $C_{\rm S}M(f,\chi)<\sqrt{2}-1$, there exists a unique solution $(u,v)\in (C^2(\O)\cap H^{2+\alpha}(\O))\times (C^2(\overline{\O})\cap H^{2+\alpha}(\O))$ to \eqref{wform} and positive $\epsilon, C$ such that for any $\cT \in \mathbb{T}(\epsilon)$ with maximal mesh size $h_{\max}$ the solution $(u_{\rm M},v_{\rm M}) \in K_{\rm M}\times \M(\T)$ to \eqref{Morleywform} is unique and satisfies 
	$$\trinl u-u_{\rm M}\trinr_{\rm pw}+ \trinl v-v_{\rm M} \trinr_{\rm pw} \le C\big( \trinl u-I_{\rm M} u\trinr_{\rm pw}+\trinl v-I_{\rm M} v\trinr_{\rm pw}+ h_{\max}\big).$$
\end{thm}

\subsection{{\em A priori} error analysis of a shifted biharmonic obstacle problem}
Let $(u,v)$ be a solution to \eqref{wform} with the regularity $u,v\in C^2(\O)\cap H^{2+\alpha}(\O)\cap H^2_0(\O)$ from Theorem \ref{uH2+delta}. The Sobolev embedding $H^{2+\alpha}(\O)\hookrightarrow W^{2,4}(\O)$ leads to $\widetilde{f}:=f+[u,v]\in L^2(\O)$.  The transformed problem seeks $u_L \in K $ such that
\begin{align}
& a(u_L,u_L-\varphi)\le (\widetilde{f},u_L-\varphi)_{L^2(\O)} \; \; \fl \varphi\in K\label{wformT}.
\end{align}
Equivalently,  $u_L$ is a minimizer in $K$ for the energy functional $J_T: H^2_0(\O) \rightarrow \R$, defined by  $J_T(\xi):=\frac{1}{2}a(\xi,\xi)- (\widetilde{f},\xi)_{L^2(\O)}$ for all $H^2_0(\O)$,
\begin{align}
& J_T(u_L)=\min_{\xi \in K} J_T(\xi)\label{energyT}.
\end{align}
\noindent\\
By construction, the solution $u$ to \eqref{wforma} also solves \eqref{wformT}, then uniqueness implies $u_L=u.$ 
\medskip

\noindent
Recall the auxiliary problem from \cite{brenner2012finite} which is a continuous problem with discrete obstacle constraints. Let $K_A:=\{\xi \in H^2_0(\Omega): \forall p\in\cV,\,\,\,\xi(p)\ge\chi(p)\}$ for the set $\mathcal{V}$ of all vertices in the triangulation $\T$. The biharmonic obstacle problem problem with discrete constraints seeks $u_A \in K_A$ such that
\begin{align}
& a(u_A,u_A-\varphi)\le (\widetilde{f},u_A-\varphi)_{L^2(\O)} \; \; \fl \varphi\in K_A\label{wformA}.
\end{align}
Equivalently, $u_A$ is a minimizer for energy functional $J_T:H^2_0(\Omega)\rightarrow\R$ over  set $K_A$,
\begin{align}
& J_T(u_A)=\min_{\xi \in K_A} J_T(\xi)  \; \fl \xi\in K_A.
\end{align}
\noindent
The solution $u_L=u$ to the biharmonic problem \eqref{wformT} and the solution $u_A$ to the corresponding auxiliary problem \eqref{wformA} satisfy the following result. 
\begin{lem}[\rm{Convergence rates} \cite{brenner2012finite}]\label{Erru_Tandu_A} 
Let $\chi\in C^2(\O)\cap C^0(\overline{\O})$ with $\max \chi(\partial\Omega)<0,$ and let $u\in C^2(\O)\cap H^{2+\alpha}(\O)$ solve \eqref{wformT}. Then, there exist an $\epsilon>0$ and $\widehat{u}_A\in K$ such that $\trinl u-u_A\trinr\lesssim h_{\max}\| u\|_{2+\alpha}$ and $\trinl \widehat{u}_A-u_A\trinr\lesssim h_{\max}^2\| u\|_{2+\alpha}$ for any triangulation $\T\in \mathbb{T}(\epsilon)$ with maximal mesh size $h_{\max}$ and the solution $u_A\in K_A$ to \eqref{wformA}.
\end{lem}	
\subsection{Proof of the main result}
\begin{proof} 
{\bf Step 1 of the proof} involves a choice of a bound for the discretization parameter for which the uniqueness of solutions to \eqref{wform}, \eqref{Morleywform} and the smallness assumption of Remark \ref{smallnessassumption} hold. Define $\mu:=C_{\rm S} M(f,\chi)<\sqrt{2}-1$ and its discrete analogue $\mu_d:=(C_{\rm S}+C(\beta)h_{\max}^{\beta}) M_{\rm d}(f,\chi) \text{ and } \mu_e:=(C_{\rm S}+C(\beta)h_{\max}^{\beta}) M(f,\chi).$ From Theorem \ref{thm.discretesobolev} it is clear that 
\begin{align}\label{muetomu}
\mu_e-\mu=C(\beta)h_{\max}^{\beta}M(f,\chi)\rightarrow 0 \text{ as } h_{\rm max}\rightarrow 0.
\end{align} This and $\mu:=C_{\rm S} M(f,\chi)<\sqrt{2}-1$ imply that there exists a positive $\epsilon_1$ for which $\mu_e:=C_{\rm dS} M(f,\chi)<\sqrt{2}-1$ for all $h_{\rm max}<\epsilon_1.$

\medskip
\noindent
Theorems \ref{thm.contsdep}, \ref{apostexistctsuniqdisc}, and \ref{thm.discretesobolev} imply
$$M^2_{\rm d}(f,\chi)-M^2(f,\chi)=C(\beta)h_{\max}^{\beta}(2C_{\rm S}+C(\beta)h_{\max}^{\beta})C^4(\psi)\|\chi\|^4_{H^2(\Omega_+)}+3C(\alpha)h_{\max}^{\alpha}\|f\|_{L^2(\O)}$$ 
and $M^2_{\rm d}(f,\chi)-M^2(f,\chi)\rightarrow 0$ as $h_{\rm max} \rightarrow 0$. This and \eqref{muetomu} imply 
$$\mu_d^2-\mu^2=(\mu^2_d-\mu^2_e)+(\mu^2_e-\mu^2)=(C_{\rm S}+C(\beta)h_{\max}^{\beta})^2(M_{\rm d}^2(f,\chi)-M^2(f,\chi))+(\mu^2_e-\mu^2)
\rightarrow 0$$
as $h_{\rm max}\rightarrow 0$. Hence, there exists a positive $\epsilon_2$ such that $\mu_d:=(C_{\rm S}+C(\beta)h_{\max}^{\beta}) M_{\rm d}(f,\chi)<\sqrt{2}-1$ for all $h_{\rm max}<\epsilon_2$.
Finally, $\mu,\mu_d,\mu_e<\sqrt{2}-1$ for any triangulation $\T\in \mathbb{T}(\epsilon :=\min\{\epsilon_1,\epsilon_2\})$. 

\medskip\noindent
The later steps of the proof focuses on the error estimates for triangulations $\T\in \mathbb{T}(\epsilon)$. Set $e:=u-u_{\rm M}$, $\delta:=v-v_{\rm M}$, and the best approximation error 
$\text{RHS}:= \trinl u-I_{\rm M} u\trinr_{\rm pw} + \trinl v-I_{\rm M}v\trinr_{\rm pw}+h_{\rm max}$.\\
{\bf Step 2 of the proof} utilizes elementary algebra to identify two critical terms 
\[T_1:= a(u,E_{\rm M} I_{\rm M} e )+  2a(v,E_{\rm M} I_{\rm M} \delta) \mbox{ and }T_2:= -(a_{\rm pw}( u_{\rm M}, I_{\rm M}e)+2a_{\rm pw}( v_{\rm M}, I_{\rm M}\delta)).\]
The definition of $a_\pw(\bullet,\bullet)$ with elementary algebra shows
\begin{align*}
\trinl e \trinr^2_{\rm pw} & = a_{\rm pw} (e, u-I_{\rm M} u ) + a_{\rm pw} (u, (1-E_{\rm M})I_{\rm M} e )
+ a(u,E_{\rm M} I_{\rm M} e ) -a_{\rm pw}( u_{\rm M}, I_{\rm M}e).
\end{align*}
Lemma \ref{Morley_Interpolation}.$a$ implies $ a_{\rm pw} (e, u-I_{\rm M} u ) = \trinl u-I_\M u\trinr_{\rm pw}^2$. The boundedness of $a_\pw(\cdot,\cdot)$, and Lemma \ref{hctenrich}.$c$-$d$ lead to
\begin{align*}
a_\pw(u,(1-E_{\rm M})I_{\rm M}e) & = a_\pw(u-I_{\rm M}u, (1-E_{\rm M})I_{\rm M}e) \le \Lambda
\trinl u-I_{\rm M}u \trinr_{\rm pw} \trinl e \trinr_{\rm pw}
\end{align*}
with $\trinl (1-E_{\rm M})I_{\rm M} e \trinr_{\rm pw} \lesssim \trinl I_{\rm M} e\trinr_{\rm pw} \le 
\trinl e \trinr_{\rm pw}.$ A combination of the previous estimates leads to
\begin{align*}
\trinl e \trinr^2_{\rm pw} & \le \trinl  u-I_{\rm M} u \trinr_\pw^2+ \Lambda
\trinl u-I_{\rm M}u \trinr_{\rm pw} \trinl e \trinr_{\rm pw}+ a(u,E_{\rm M} I_{\rm M} e ) -a_{\rm pw}( u_{\rm M}, I_{\rm M}e).
\end{align*}
The analogous result with $(u,u_{\rm M},e)$ replaced by $(v,v_{\rm M}, \delta)$ reads
\begin{align*}
\trinl \delta \trinr^2_{\rm pw} &  \le \trinl  v-I_{\rm M} v \trinr_\pw^2+ \Lambda
\trinl v-I_{\rm M}v \trinr_{\rm pw} \trinl \delta\trinr_{\rm pw}+ a(v,E_{\rm M} I_{\rm M} \delta ) -a_{\rm pw}( v_{\rm M}, I_{\rm M}\delta).
\end{align*}
A weighted sum of those two estimates plus the Cauchy-Schwarz inequality shows
\begin{align}\label{estt1t2}
\trinl e \trinr^2_{\rm pw} + 2 \trinl \delta \trinr^2_{\rm pw} & \le 2 \; \text{RHS}^2 + \Lambda (\trinl u-I_{\rm M}u \trinr_{\rm pw} \trinl e\trinr_{\rm pw} + 2  \trinl v-I_{\rm M}v \trinr_{\rm pw} \trinl \delta \trinr_{\rm pw}  ) +T_1+T_2.
\end{align}

\noindent
{\bf Step 3 of the proof} employs three variational inequalities: The test function $\widehat{u}_A \in K$ from Lemma \ref{Erru_Tandu_A} leads in \eqref{wforma} to \[a(u,u-\widehat{u}_A)+2b(u,v,u-\widehat{u}_A)\le (f,u-\widehat{u}_A)_{L^2(\O)}.\]
The test function $\varphi=E_\M u_\M \in K_A$ and the definition of $b(\bullet,\bullet,\bullet)$ lead in \eqref{wformA} to 
\[a(u_A,u_A-E_\M u_\M)\le (f,u_A-E_\M u_\M)_{L^2(\O)}-2b(u,v,u_A-E_\M u_\M).\] 
\noindent
The test functions $\varphi_1=I_{\rm M} u\in K_{\rm M}$ and $\varphi_2=I_\M \delta\in V_\M$ lead in \eqref{Morleywform} to
\[ a_\pw(u_\M,u_\M-I_{\rm M} u)+ 2b_\pw(u_\M,u_\M-I_{\rm M} u, v_\M)\le (f,u_\M-I_{\rm M} u)_{L^2(\O)},\]
\[a_\pw(v_\M,I_\M \delta)-b_\pw(u_\M,u_\M,I_\M \delta)=0.\]
\noindent
The sum of preceeding four displayed estimates lead to one inequality with many terms. 
An elementary, but lengthy algebra leads to an estimate for $(T_1+T_2)/2$ and the crucial terms $b_\pw(e,e,v)-b_\pw(e,u,\delta).$ The following list of identities are employed in the calculation:
$a_\pw(I_{\rm M}u,E_{\rm M}I_{\rm M} u-u)=b_\pw(u_{\rm M},u_{\rm M},v_{\rm M}-E_{\rm M}v_{\rm M})=0$ from  Remark \ref{remInterEnt}; $b_\pw(u_\M, I_\M u-u,\Pi_0v)=0$ from  Lemma \ref{Morley_Interpolation}.$a$ where, $\Pi_0 v\in {\mathcal P}_0(\mathcal{T})$ is the piecewise constant $L^2$ projection of $v$; $b(u_\M,u_\M,E_\M v_\M-v_\M)=b_\pw(I_\M u,I_\M u,v_\M-E_\M v_\M)=0$ from  Lemma \ref{hctenrich}.$b$; $b_\pw(I_\M u,u_\M -E_\M u_\M,\Pi_0 v)=0$ from Lemma \ref{hctenrich}.$c$; and the symmetry $b(u,u,v)=b(u,v,u)$. The resulting inequality is equivalent to
\begin{align}
&\frac{T_1+T_2}{2}+b_\pw(e,e,v)-b_\pw(e,u,\delta)\nonumber\\
&\quad \le \half a_\pw(u-I_{\rm M}u,E_{\rm M}I_{\rm M} u-u)+a_\pw(v-I_{\rm M}v,E_{\rm M}I_{\rm M} v-v)\nonumber\\
&\qquad +\half a(u,\widehat{u}_A-u_A)+a(u-u_A,u_A-E_\M u_\M)-b(u,v,u_A-\widehat{u}_A)\nonumber\\
&\qquad + b_\pw(u-I_\M u,(1 -E_\M) u_\M,v)+b_\pw(I_\M u,(1 -E_\M) u_\M,v-\Pi_0v) \nonumber\\
&\qquad +b_\pw(u-I_\M u,u,(1 -E_\M) v_\M)+b_\pw(I_\M u,u-I_\M u,(1 -E_\M) v_\M)\nonumber\\
&\qquad +b_\pw(u_\M, I_\M u-u,v_\M-v)+b_\pw(u_\M, I_\M u-u,v-\Pi_0v)\nonumber\\
&\qquad +b_\pw(u_\M,u_\M,v-I_\M v)+\half (f,u-\widehat{u}_A+u_A-E_\M u_\M-I_\M e)_{L^2(\O)}\nonumber\\
&\qquad =:T_3+\dots+T_{15}.\label{t1t2}
\end{align}

\noindent {\bf Step 4 of the proof} estimates of the terms $T_3,\dots,T_{15}$ on the right-hand side of \eqref{t1t2} and establishes the bound $C(h_{\max}^{2}+\trinl u-I_\M u\trinr_{\rm pw}^2+\trinl v-I_\M v\trinr_{\rm pw}^2+(\trinl e\trinr_\pw +\trinl \delta\trinr_\pw )\text{RHS})$ with a constant $C\approx 1$ that depends on $\trinl u\trinr, \trinl v\trinr, \| u\|_{2+\alpha}, \| v\|_{2+\alpha}, \|f\|_{L^2(\Omega)}, \Lambda$ and is independent of $h_{\max}$. Elementary algebra lead to first equality in 
$$T_3=a_\pw(u-I_{\rm M}u,(E_{\rm M}-1)I_{\rm M} u)-\trinl u-I_\M u\trinr_{\rm pw}^2\le (\Lambda-1)\trinl u-I_\M u\trinr_{\rm pw}^2,$$
with the Cauchy-Schwarz inequality and Lemma \ref{hctenrich}.$d$ in the final step. The analysis of $v$ replaced by $u$ as in the estimate of $T_3$ reads 
$$T_4=a_\pw(v-I_{\rm M}v,E_{\rm M}I_{\rm M} v-v)\le (1+\Lambda)\trinl v-I_\M v\trinr_{\rm pw}^2.$$ 

\noindent Lemma \ref{Erru_Tandu_A} and the Cauchy-Schwarz inequality show $$T_5+T_7=\half a(u,\widehat{u}_A-u_A)-b(u,v,u_A-\widehat{u}_A)\lesssim h^2_{\max}\| u\|_{2+\alpha}(\trinl u\trinr+\|[u,v]\|_{L^2(\O)}).$$
The Cauchy-Schwarz inequality, a triangle inequality, Lemma \ref{Erru_Tandu_A}, and Lemma \ref{hctenrich}.$d$ lead to
\begin{align*}
T_6=a(u-u_A,u_A-E_\M u_\M)&\le \trinl u-u_A\trinr\trinl u_A-E_\M u_\M\trinr
\lesssim h_{\max}\| u\|_{2+\alpha}\big(h_{\max}\| u\|_{2+\alpha}+(1+\Lambda)\trinl e\trinr_{\rm pw}\big).
\end{align*}
The boundedness of $b_\pw(\cdot,\cdot,\cdot)$, Lemma \ref{hctenrich}.$d$, a triangle inequality, and $\|v-\Pi_0 v\|_{L^{\infty}(\O)}\lesssim h_{\max}\| v\|_{1,\infty}\lesssim h_{\max}\| v\|_{2+\alpha} $ lead to 
\begin{align*}
T_8+T_9&=b_\pw(u-I_\M u,(1 -E_\M)u_\M,v)+b_\pw(I_\M u,(1 -E_\M) u_\M,v-\Pi_0v) \nonumber\\
&\lesssim 
\trinl e\trinr_\pw\big(\trinl u-I_\M u\trinr_\pw \trinl v\trinr+h_{\max}\trinl u\trinr\| v\|_{2+\alpha}\big). 
\end{align*}	
Lemma \ref{hctenrich}.$d$ and Lemma \ref{Bound bpw} imply
\begin{align*}
T_{10}+T_{11}&=b_\pw(u-I_\M u,u,(1 -E_\M) v_\M)+b_\pw(I_\M u,u-I_\M u,(1 -E_\M) v_\M)\nonumber\\
&\lesssim 
\trinl u \trinr\trinl  u-I_\M u\trinr_{\pw} \trinl\delta \trinr_\pw.
\end{align*}
Lemma \ref{Bound bpw} and a piecewise Poincar\'e inequality show
\begin{align*}
T_{12}+T_{13}&=b_\pw(u_\M, I_\M u-u,v_\M-v)+b_\pw(u_\M, I_\M u-u,v-\Pi_0v)\nonumber\\
&\lesssim 
\trinl u_\M \trinr_\pw \trinl  u-I_\M u\trinr_{\pw}\big(\trinl \delta \trinl_\pw +h_{\max} \| v\|_{2+\alpha}\big).
\end{align*}
Remark \ref{remInterEnt}, Lemma \ref{hctenrich}.$a$, a generalised H\"older inequality, interpolation estimate \cite[(6.1.5)]{ciarlet1978finite}, and Lemma \ref{hctenrich}.$d$ lead to
\begin{align*}
T_{14}&=b_\pw(u_\M,u_\M,v-E_\M I_\M v)=b_\pw(u_\M,u_\M,(v-E_\M I_\M v)-I_\M(v-E_\M I_\M v)) \\
&\le \trinl u_\M \trinr_\pw^2\| (v-E_\M I_\M v)-I_\M(v-E_\M I_\M v)\|_{L^\infty(\O)}\\
&\lesssim  h_{\max}\trinl u_\M \trinr_\pw^2\trinl v-E_\M I_\M v \trinr_\pw \lesssim  h_{\max}\trinl u_\M \trinr_\pw^2\trinl v- I_\M v \trinr_\pw.
\end{align*}
The Cauchy-Schwarz inequality, Lemma \ref{Erru_Tandu_A}, Lemma \ref{Morley_Interpolation}.$b$-$c$, a triangle inequality and Lemma \ref{hctenrich}.$d$ imply
\begin{align*}
T_{15}=&(f,u-\widehat{u}_A+u_A-E_\M u_\M-I_\M e)_{L^2(\O)}=(f,u_A-\widehat{u}_A)_{L^2(\O)}\nonumber\\
&+(f,(1-I_\M)(u-E_\M u_\M))_{L^2(\O)}\lesssim  h_{\max}^2\|f\|_{L^2(\O)}\big(\| u\|_{2+\alpha}+\trinl e \trinr_{\rm pw}\big).
\end{align*}
{\bf Step 5 of the proof} estimates the $b(\bullet,\bullet,\bullet)$ terms on the left-hand side of \eqref{t1t2}. For any $\cT\in \mathbb{T}(\epsilon)$, it is clear from Step 1 that $0 <\mu_e:= C_{\rm dS}M(f, \chi) < \sqrt{2}-1$. It is elementary to see that this implies $\mu_e < \frac{1}{\mu_e} -2$. Hence there exist a real $\gamma$ with $\mu_e < \gamma< \frac{1}{\mu_e} -2$. Lemma \ref{Bound bpw}, Young inequality, and {\em a priori} bound from Theorem \ref{thm.contsdep} imply
\begin{align*}
&b_\pw(e,u,\delta)-b_\pw(e,e,v)\le C_{\rm dS}\trinl u \trinr \trinl e \trinr_\pw \trinl \delta \trinr_\pw+C_{\rm S} \trinl e \trinr_\pw^2\trinl v \trinr\nonumber\\
&\qquad \le \mu_e \big(\sqrt{2}\trinl e \trinr_\pw \trinl \delta \trinr_\pw+ \trinl e \trinr_\pw^2 \big) \le \mu_e \big((1+ \frac{\gamma}{2})\trinl e \trinr_\pw^2 + \gamma^{-1}  \trinl \delta \trinr_\pw^2 \big).
\end{align*}
{\bf Step 6 of the proof} combines all the estimates $T_3,\dots,T_{15}$ of \eqref{t1t2} with \eqref{estt1t2}. Set positive constants $\tau_1:=C(1-\mu_e (\gamma+2))$ and $\tau_2:= 2C(1-\mu_e \gamma^{-1})$ and deduce
\begin{align*}
\tau_1\trinl e \trinr_\pw^2+ \tau_2\trinl \delta \trinr_\pw^2&\le  h_{\max}^{2}+\trinl u-I_\M u\trinr_{\rm pw}^2+\trinl v-I_\M v\trinr_{\rm pw}^2+(\trinl e\trinr_\pw +\trinl \delta\trinr_\pw )\text{RHS},
\end{align*}
with some universal constant $C\approx 1$ from the estimates of $T_3,\dots,T_{15}$. The Young inequality for the last term on the right-hand side of the above estimate imply the assertion
\begin{align*}
\tau_1\trinl e \trinr_\pw^2+ \tau_2\trinl \delta \trinr_\pw^2&\le 2(h_{\max}^{2}+\trinl u-I_\M u\trinr_{\rm pw}^2+\trinl v-I_\M v\trinr_{\rm pw}^2)+\frac{\tau_1+\tau_2}{\tau_1\tau_2} {\text{RHS}}^2.
\end{align*}
This concludes the proof.
\end{proof} 

\section{Implementation Procedure and Numerical Results}\label{sec.numericalresults}
The first subsection is devoted to the implementation procedure to solve the discrete problem \eqref{Morleywform}. Subsections \ref{eg.square} and \ref{eg.Lshaped} deal with the results of the numerical experiments and is followed by a subsection on conclusions.
\subsection{Implementation procedure}\label{sec.implementation}
\noindent  
The solution $(u_\M,v_\M)$ to \eqref{Morleywform} is computed using a combination of Newtons' method \cite{mallik2016conforming} in an inner loop and primal dual active set strategy \cite{hintermuller2002primaldual} in an outer loop. The initial value $u_\M^{init}$ for $u_\M$ in the iterative scheme is the discrete solution to the biharmonic obstacle problem: seek $u_\M^{init} \in K(\chi,\T)$ such that
\begin{equation}
	a_\pw(u_\M^{init},u_\M^{init}-\varphi_\M)\le (f,u_\M^{init}-\varphi_\M)_{L^2(\O)}   \; \; \fl \varphi_\M\in K(\chi,\T)\label{MorleybiharmonicOP}
\end{equation}
with $K(\chi,\T)$ from Subsection \ref{sec.preliminaries}. Since \eqref{MorleybiharmonicOP} is \eqref{Morleywforma} without the trilinear term, $u_\M^{init}$ is computed with the same algorithm below without the inner loop for the nonlinearity. This is shown in Figure \ref{flowchart} and the general case is described in the sequel. 
\medskip

\noindent
Recall $\M(\T)$, $\mathcal V$ and $\mathcal E$ from Subsection \ref{sec.preliminaries}. Let $p \in \mathcal V(\O)$ and $(\varphi_1,\dots, \varphi_N)$ be the node- and edge-oriented basis functions in $\M(\T)$, $N:=|\mathcal V(\O)|+|\mathcal E(\O)|$; see \cite{carstensen2014discrete} for details and basic algorithms for the Morley FEM. Let $u_\M=\sum_{j=1}^{N}{{\alpha}}_j\varphi_j$ and $v_\M=\sum_{j=1}^{N}{{\beta}}_j\varphi_j$ with ${\boldsymbol{\alpha}}=(\alpha_1,\dots,\alpha_N)$ and ${\boldsymbol{\beta}}=(\beta_1,\dots,\beta_N).$

\smallskip

\noindent $ \textit{ Primal dual active set strategy}. $ 
\begin{itemize}
	\item Choose initial values $(u_\M^0,v_\M^0)=(u_\M^{init},0)$. 
	\item In the $m$th step of the primal dual active set algorithm, find the active $Ac^m$ and inactive $In^m$ sets defined by
	\begin{subequations}\label{active.inactive}
		\begin{align}
			Ac^{m}=\{p \in \mathcal V(\O): {\boldsymbol{\lambda}}^{m-1}(p)+\chi(p)-u_\M^{m-1}(p)&\le 0\},\\
			In^{m}=\{p \in \mathcal V(\O): {\boldsymbol{\lambda}}^{m-1}(p)+\chi(p)-u_\M^{m-1}(p)&>0\}.
		\end{align}
	\end{subequations} 
	Since the degrees of freedom also involve the midpoints of the interior edges, let $I^m:=In^m \cup \mathcal{E}(\O)$ be the union of $In^m$ and the midpoints of interior edges.
	
	\smallskip
	\noindent $(a) \textit{ Non-linear system}.$ 
	\begin{itemize}
		\item The matrix formulation corresponding to \eqref{Morleywform} can be expressed as block matrices in term of active and inactive sets and load vector $\bf{F}$ on the right-hand side.
		\item Impose $u_\M^m(Ac^m)=\chi(Ac^m)={\boldsymbol{\alpha}}^m(Ac^m)$ and ${\boldsymbol{\lambda}}^m(I^m)=0$. From here on, superscript $m$ is omitted and $({\boldsymbol{\alpha}}(I^m),{\boldsymbol{\lambda}}(Ac^m),{\boldsymbol{\beta}}(Ac^m),{\boldsymbol{\beta}}(I^m))$ is replaced by $({\boldsymbol{\alpha}}_2, {\boldsymbol{\lambda}}_1,{\boldsymbol{\beta}}_1,{\boldsymbol{\beta}}_2)$ for notational convenience. 
		\item After substitution of the known values ${\boldsymbol{\alpha}}(Ac)$ and ${\boldsymbol{\lambda}}(I)=0$, the discrete problem reduces to a smaller non-linear system of equations ${\bf{G}}({\boldsymbol{\alpha}}_2, {\boldsymbol{\lambda}}_1,{\boldsymbol{\beta}}_1,{\boldsymbol{\beta}}_2)={\bf{0}}$. 

	\end{itemize} 
	
	\smallskip 
	\noindent $(b) \textit{ Newtons iteration with initial guess } {\bf{ S}}^0=({\boldsymbol{\alpha}}_2^0, {\boldsymbol{\lambda}}_1^0,\boldsymbol{0},\boldsymbol{0}).$ 
	\begin{itemize}
		\item For ${\bf{ S}}^{n}:=({{\boldsymbol{\alpha}}}_2^{n},{\boldsymbol{\lambda}}_1^{n}, {\boldsymbol{\beta}}_1^{n},
		{\boldsymbol{\beta}}_2^{n})$,
		do ${\bf{S}}^{{n}+1} = {\bf{S}}^{n} - {\bf{\Delta}\bf { S}}^{n}$ for the solution ${\bf{\Delta} \bf{S}}^{n}$ of the linear system of equations ${\bf{J_{\bf{G}}}}({\bf{S}}^{n}){\bf{\Delta}S}^{n}={\bf{G}}({\bf{S}}^{n})$ with $\bf{J_G}$ is the Jacobian matrix of ${{\bf{G}}}$ until $\|{\bf{\Delta} {S}}^{n}\|_{l^2(\R^{2N})}$ is less than a given tolerance.

	\end{itemize}
	\item Update $m=m+1$. This primal-dual active strategy iteration procedure terminates when $Ac^m=Ac^{m-1}$ and $I^m=I^{m-1}$. 
\end{itemize}

\medskip \noindent The flowchart below (see, Figure \ref{flowchart}) demonstrates the combined primal-dual active set and Newton algorithms for $\T_0,\T_1,\dots$.
\noindent
\tikzstyle{inputoutput1} = [trapezium, draw, fill=white!20, 
text width=10em, trapezium left angle=60, trapezium right angle=120,text centered, rounded corners, minimum height=1em]
\tikzstyle{inputoutput} = [trapezium, draw, fill=white!20, 
text width=13em, trapezium left angle=60, trapezium right angle=120,text centered, rounded corners, minimum height=1em]
\tikzstyle{decision} = [diamond, draw, fill=white!20, 
text width=4em, text badly centered, node distance=2cm, inner sep=0pt]
\tikzstyle{block} = [rectangle, draw, fill=white!20, 
text width=13em, text centered, rounded corners, minimum height=1em]
\tikzstyle{block1} = [rectangle, draw, fill=white!20, 
text width=15em, text centered, rounded corners, minimum height=1em]
\tikzstyle{line} = [draw, -latex']
\tikzstyle{cloud} = [draw, ellipse,fill=white!20, node distance=3cm,
minimum height=2em]
\begin{figure}[h]
	\begin{tikzpicture}[node distance = 1.3cm, auto]
	\tikzstyle{every node}=[font=\scriptsize]
	\node [inputoutput,node distance=1cm] (init){Initialize $m$, err, $\rho:=1$ and\\ ${\boldsymbol{\alpha}}^0,{\boldsymbol{\beta}}^0, {\boldsymbol{\lambda}}^0, {n}:=0$}; 
	\node [block, below of=init, node distance=1.25cm] (compute) { Assemble the matrices in ${\bf{J_G}}$ and $\bf{F}$;\\ Compute $Ac^{1}$ and $I^1$ from \eqref{active.inactive}};
	\node [decision, below of=compute,node distance=2cm] (decide) { err>$10^{-7}$?};
	\node [block1, right of=decide, node distance=5cm] (set) {${\boldsymbol{\alpha}}^m(Ac^m):=\chi(Ac^m)$, ${\boldsymbol{\lambda}}^m(I^m):=0$; Solve \eqref{MorleybiharmonicOP} for ${\boldsymbol{\alpha}}^m(I^m)$, ${\boldsymbol{\lambda}}^m(Ac^m)$; \\Update err:=$\|\boldsymbol{\alpha}^m-\boldsymbol{\alpha}^{m-1}\|_{\infty}$, $m:=m+1$};
	\node [decision, right of=set, node distance=5cm] (break) {$Ac^{m-1}=Ac^{m}$?};
	\node [block1, below of=decide, node distance=2cm] (set1) {${\boldsymbol{\alpha}}^0:={\boldsymbol{\alpha}}^{m-1}$, ${\boldsymbol{\lambda}}^0:={\boldsymbol{\lambda}}^{m-1}$, $m$, err := 1;\\Compute $Ac^{1}$ and $I^1$ from \eqref{active.inactive}};
	\node [decision, below of=set1,node distance=2cm] (decide1) {err>$10^{-7}$?}; 
	\node [block, right of=decide1, node distance=4.3cm] (set2) {${\boldsymbol{\alpha}}^m(Ac^m):=\chi(Ac^m)$, ${\boldsymbol{\lambda}}^m(I^m):=0$};
	\node [decision, below of=set2,node distance=1.9cm] (decide2) { $\rho>$ $10^{-7}$?};
	\node [block, below of=decide2,node distance=2.2cm] (compute2) { Update ${\bf{S}}^{{n}+1} = {\bf{S}}^{n}-{\bf{\Delta}\bf { S}}^{n}$, $\rho:=\|{\bf{\Delta} {S}}^{n}\|_{l^2(\R^{2N})}$, $n:=n+1$ };
	\node [block, right of=decide2, node distance=4.5cm] (compute3) {Update $\boldsymbol{\alpha}^m(I^m)={\bf{S}}^{n}(1:{\rm length}(I^m))$, err:=$\|\boldsymbol{\alpha}^m-\boldsymbol{\alpha}^{m-1}\|_{\infty}$, $m:=m+1$};
	\node [decision, below of=compute3, node distance=3cm] (break1) {$Ac^{m-1}=Ac^{m}$?};
	\node [block, below of=compute2, node distance=2.3cm] (output){Compute $u_\M$, $v_\M$ using the basis functions, ${\boldsymbol{\alpha}}^{m-1}$, and ${\boldsymbol{\beta}}^{m-1}$};

	\path [line] (init) -- (compute);
	\path [line] (compute) -- (decide);
	\path [line] (decide) -- node [near start] {yes}(set);
	\path [line] (set) -- (break);
	\path [line] (break) |- node[xshift=0.3cm]{no}(0,-2)--(decide.north);
	\path [line] (decide) -- node [near start] {no}(set1);
	\path [line] (break) |- node [near start] {yes}(set1);
	\path [line] (set1) -- (decide1);
	\path [line] (decide1) -- node [near start] {yes}(set2);
	\path [line] (set2) -- (decide2);
	\path [line] (decide2) --node [near start] {yes} (compute2);
	\path [line] (compute2.west) --++(-1,0)|- (decide2.west);
	\path [line] (decide2) -- node [near start] {no}(compute3);
	\path [line] (compute3) -- (break1);
	\path [line] (break1) -| node[xshift=3.5cm] [near start] {no}(decide1);
	\path [line] (break1) |- node [near start] {yes}(output);
	\path [line] (decide1.west)--++(-1,0)|- node[yshift=3.45cm,xshift=0.3cm] [near start] {no}(output);
	\end{tikzpicture}
	\caption{Flowchart for the primal-dual active set strategy with the Newtons method}\label{flowchart}
\end{figure}
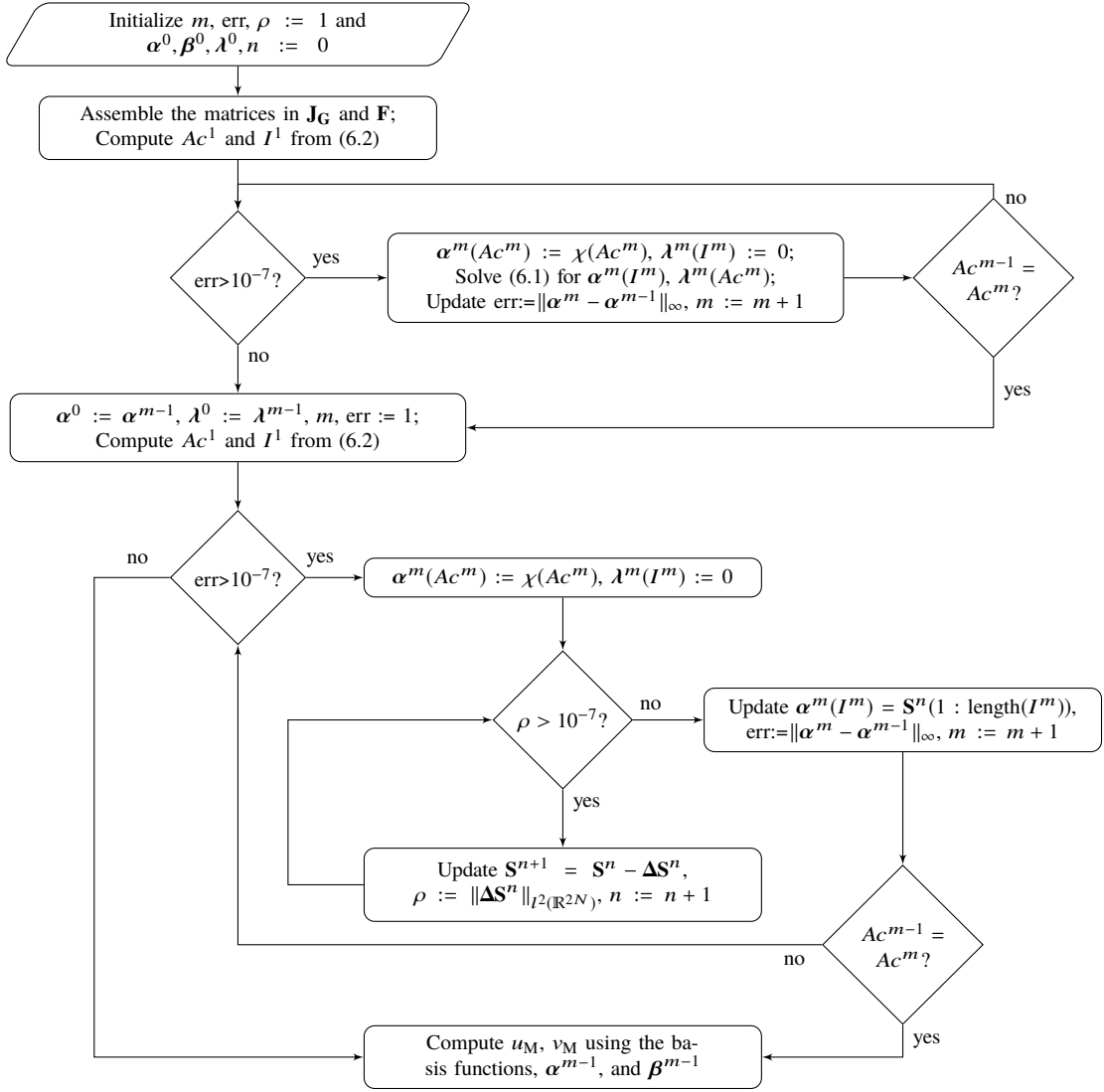

\noindent We observe in the examples of this paper (for small $f$ and $\chi$) that at each iteration of primal dual active set algorithm, the Newtons' method converges in four iterations. In this case, we notice that the error between final level and the previous level of the nodal and edge-oriented values in Euclidean norm of $\R^{2N}$ is less than $10^{-9}$. Also, the primal dual active set algorithm terminates within three steps.
\medskip

\noindent 
The uniform mesh refinement has been done by red-refinement criteria, where each triangle is subdivided into four sub-triangles by connecting the midpoints of the edges. Let $u_\ell$ (resp.$v_\ell$) be the discrete solution at the $\ell$th level for $\ell=1,2,3,..$, $L$ and define
$$e_\ell(u):=\trinl u_L-u_\ell \trinr_{\pw} \mbox{ and }  \widetilde{e}_\ell(u):=\max_{p \in \mathcal{V}_{\ell}} |u_L(p)-u_\ell(p)|.$$
$$\Big(\mbox{resp. }e_\ell(v):=\trinl v_L-v_\ell \trinr_{\pw} \mbox{ and }  \widetilde{e}_\ell(v):=\max_{p \in \mathcal{V}_{\ell}} |v_L(p)-v_\ell(p)|\Big).$$
The order in $H^2$ norm (resp. $L^\infty$ norm) at $\ell$th level for $u$ is approximated by ${\rm{EOC(\ell)}}:={\rm{log}}\big({e_\ell(u)}/{e_{L-1}(u)}\big)/{\rm{log}}(2^{L-1-\ell})$ (resp. ${\rm{log}}\big({\widetilde{e}_\ell(u)}/{\widetilde{e}_{L-1}(u)}\big)/{\rm{log}}(2^{L-1-\ell})$) for $\ell=1,\dots,L-2$. The discrete coincidence set is $\mathcal{C}_\ell:=\big\{p \in \mathcal{V}_\ell;u_\ell(p)-\chi(p)\le\widetilde{e}_\ell(u) \big\}$ for the level $\ell$.

\subsection{The von K\'{a}rm\'{a}n obstacle on the square domain}\label{eg.square}
Let the computational domain be $\Omega=0.5(-1,1)^2$. The criss-cross mesh with $h=1$ is taken as the initial triangulation $\cT_0$ of $\Omega$.  Consider the \vk\, obstacle problem \eqref{wform} for the three examples in this section. Examples 1 and 2 take $f=0$ with different obstacles; Example 3 concerns a significantly huge function $f$. 


%
%

\medskip
\noindent {\bf Example 1 (Coincidence set with non-zero measure). }Let the obstacle be given by $\chi(\boldsymbol{x})=1-5|\boldsymbol{x}|^2+|\boldsymbol{x}|^4, \boldsymbol{x} \in \Omega=0.5(-1,1)^2.$ This example is taken from \cite{brenner2013morley}. The discrete coincidence $\mathcal{C}_6$ and $\mathcal{C}_7$ are displayed in Figure \ref{fig.eg1}. Since $\Delta^2 \chi =64>0$ in this example, it is known from \cite[Section 8]{caffarelli1979obstacle} that the non-coincidence set $\O \setminus \mathcal{C}$ is connected. This behaviour of the non-coincidence set can be seen in Figure \ref{fig.eg1} for levels 6 and 7.
\begin{table}[h!!]  
	\caption{\small{Convergence results for Example 1 on the square domain}}
	{\small{\scriptsize
			\begin{center}
				\begin{tabular}{ ||c|c||c|c||c|c||c | c || c|c||c|c||}
					\hline
					$\ell$ &$h$     & $  \widetilde{e}_\ell(u)$ & EOC &  $  \widetilde{e}_\ell(v)$ & EOC &$e_\ell(u)$ & EOC  &$e_\ell(v)$ & EOC  \\	
					\hline
					1&0.5000&	0.013222&1.2098& 0.125162&1.9151& 16.496069&0.7666& 1.409870&0.9561\\
					2&0.2500& 0.013222&1.5123&0.045884& 2.0319& 12.963642&0.8714&1.025239& 1.0802\\ 
					3&0.1250&0.011327&1.9419& 0.012143&2.0699 &8.621491 &0.9657&0.493374&1.0885\\
					4&0.0625&0.003404&2.0456& 0.003205&2.1440  &4.927900 &1.0450& 0.235687&1.0999\\ 
					5&0.0313&0.000909&2.1862 & 0.000808&2.3000&2.541191& 1.1345& 0.114679& 1.1605\\
					6&0.0156&0.000200&-&0.000164&-& 1.157459& -&0.051304 &- \\
					\hline	
				\end{tabular}
			\end{center}
		}}\label{table.eg1}
	\end{table}
	\begin{figure}[h]
		\begin{center}
			\begin{minipage}[H]{0.45\linewidth}
				{\includegraphics[width=13cm]{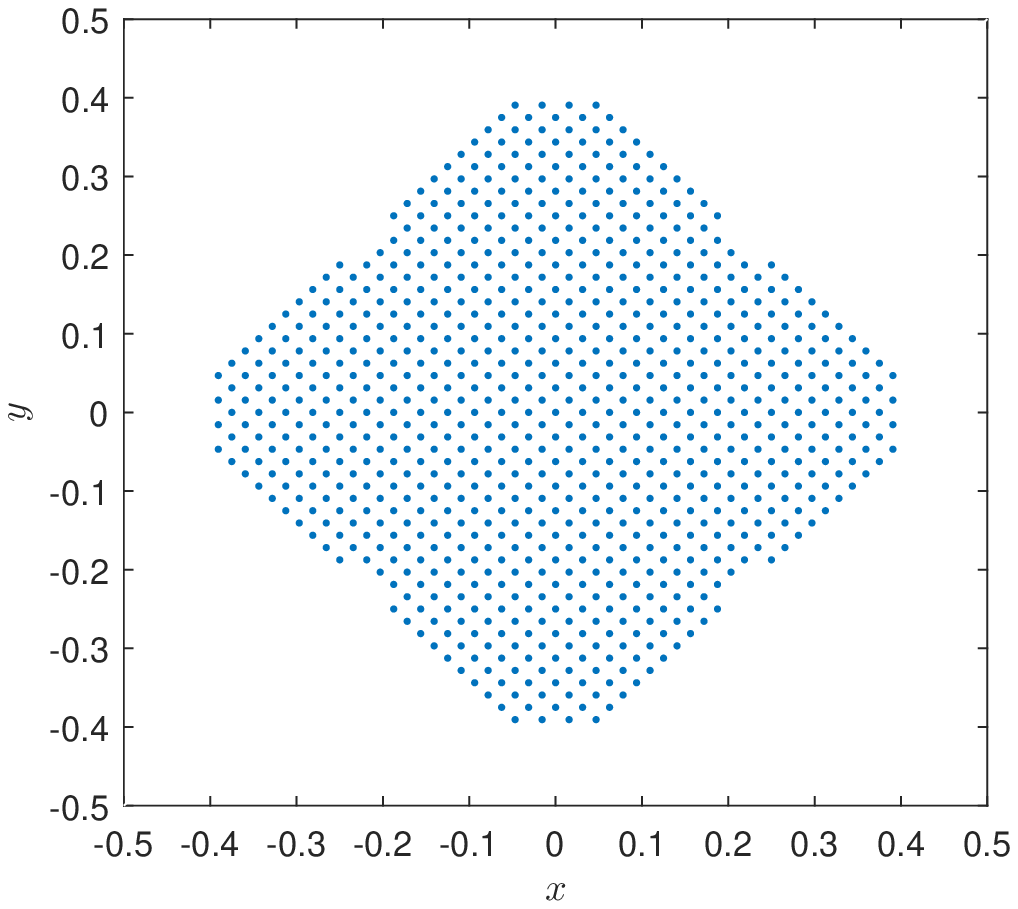}}
			\end{minipage}
			\begin{minipage}[H]{0.45\linewidth}
				{\includegraphics[width=13cm]{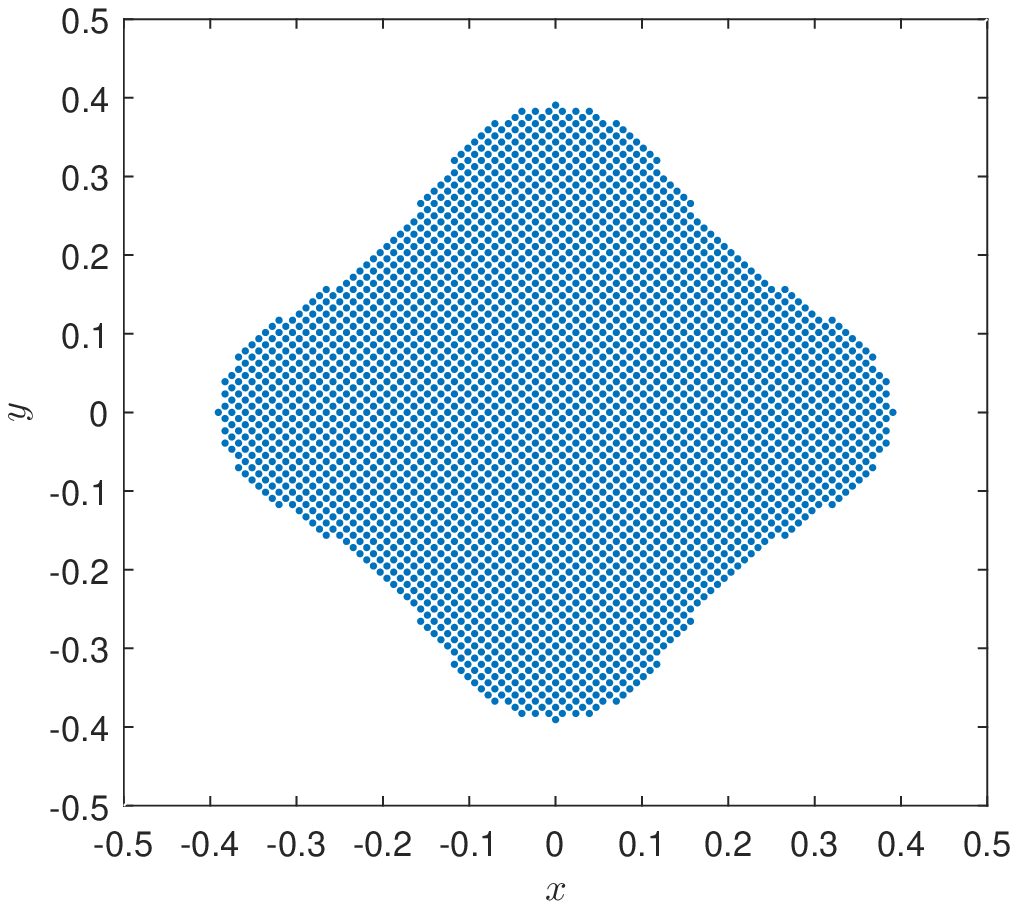}}
			\end{minipage}
			\caption{$\mathcal{C}_6$ and $\mathcal{C}_7$, Example 1}\label{fig.eg1}
		\end{center}
	\end{figure}	
	
	\noindent Table \ref{table.eg1} shows  errors and orders of convergence for the displacement $u$ and the Airy-stress function $v$. Observe that linear order of convergences are obtained for $u$ and $v$ in the energy norm, and quadratic order of convergence in $L^\infty$  norm. These numerical order of convergence in energy norm clearly matches the expected order of convergence given in Theorem \ref{thm.err}. Though the theoretical rate of convergence in $L^\infty$ norm is not analysed, the numerical rates are obtained similar to that in \cite{brenner2013morley} for the biharmonic obstacle problem.
	
	\medskip
	
	\noindent {\bf Example 2. (Coincidence set with zero measure) } In this example taken from \cite{brenner2013morley}, $\chi(\boldsymbol{x})=1-5|\boldsymbol{x}|^2-|\boldsymbol{x}|^4, \boldsymbol{x} \in \Omega=0.5(-1,1)^2$ with $\Delta^2 \chi =-64<0$ in $\O$, and hence, the interior of the coincidence set must be empty, since $\Delta^2 u$ (in the sense of distributions)
	is a nonnegative measure (\cite[Section 8]{caffarelli1979obstacle}). This can be observed in the pictures of the discrete coincidence sets displayed in Figure \ref{fig.eg2}.
	\begin{table}[h!!] 
		\caption{\small{Convergence results for Example 2 on the square domain}}
		{\small{\scriptsize
				\begin{center}
					\begin{tabular}{ ||c|c||c|c||c|c||c | c || c|c||c|c||}
						\hline
						$\ell$ &$h$     & $  \widetilde{e}_\ell(u)$ & EOC & $  \widetilde{e}_\ell(v)$ & EOC &$e_\ell(u)$ & EOC &$e_\ell(v)$ & EOC  \\	\hline
						1&0.5000&0.028792&1.4917&0.136864& 1.8793 & 15.510398&0.7999& 1.493256&0.9636\\
						2&0.2500&0.028792&1.8646&0.050539&1.9898& 11.837363&0.9024& 1.070278 &1.0843\\ 
						3&0.1250&0.009347& 1.9451& 0.014530&2.0535 &7.563740&0.9878& 0.510661&1.0899\\ 
						4&0.0625&0.003116&2.1252&0.003980&2.1462& 4.210097& 1.0591& 0.244868& 1.1047\\ 
						5&0.0312&0.000843&2.3636&0.001030&2.3427 & 2.138703& 1.1411&0.118649& 1.1642\\ 
						6&0.0156&0.000164&-&0.000203&-& 0.969687&-&0.052944&-\\
						\hline				
					\end{tabular}
				\end{center}
			}}\label{table.eg2}
		\end{table}
		\begin{figure}[h]
			\begin{center}
				\begin{minipage}[H]{0.4\linewidth}
					{\includegraphics[width=12cm]{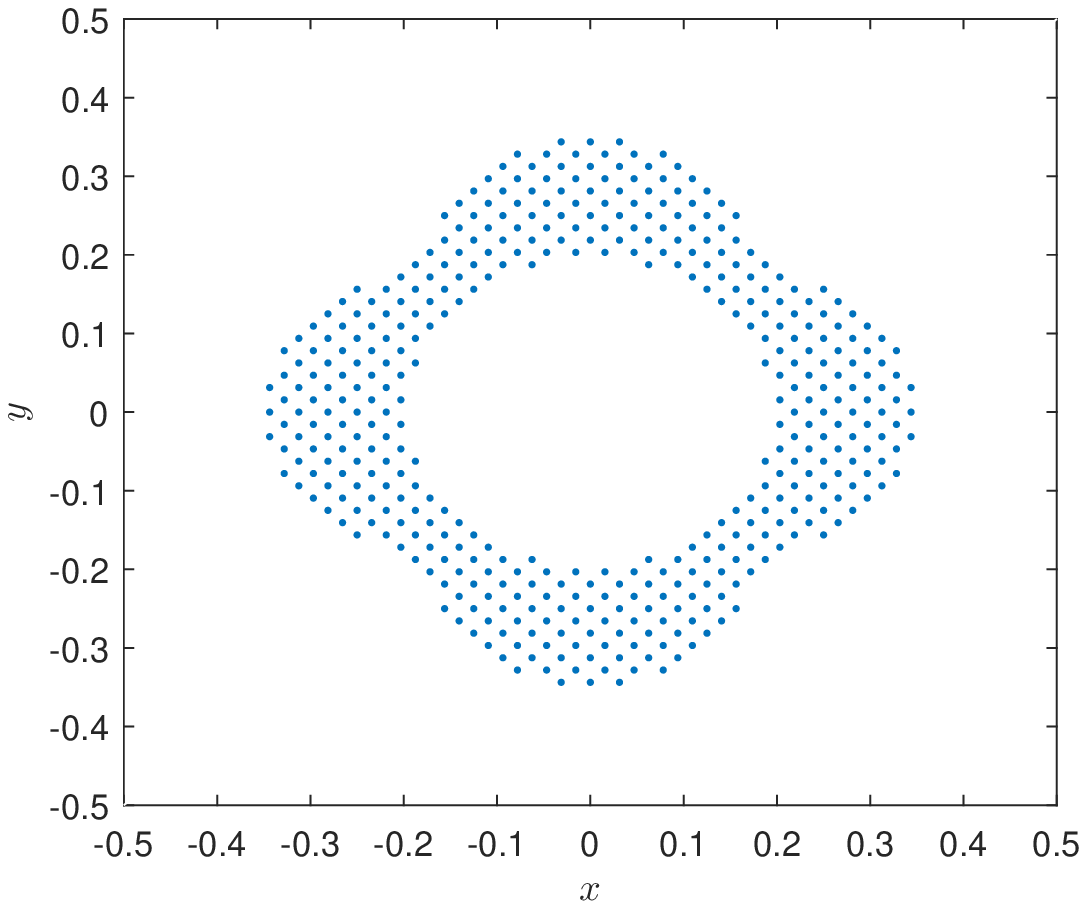}}
				\end{minipage}
				\begin{minipage}[H]{0.4\linewidth}
					{\includegraphics[width=12cm]{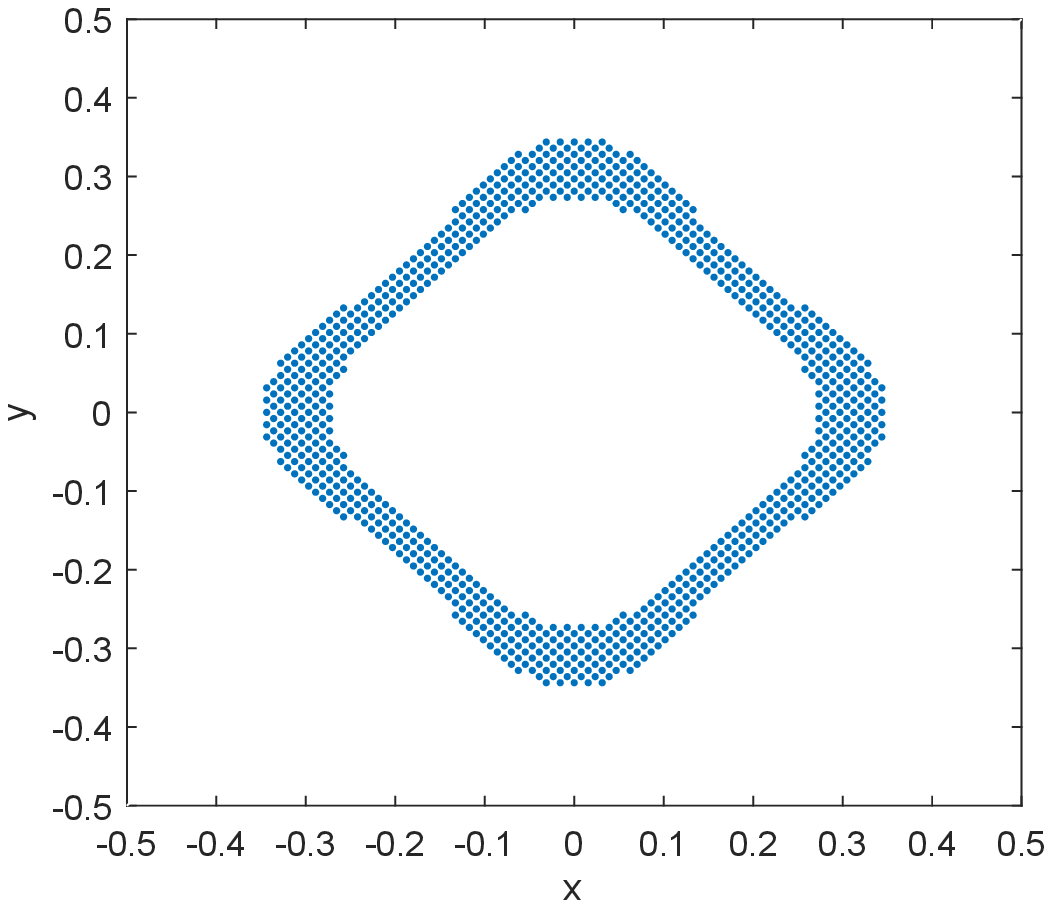}}
				\end{minipage}
				\caption{$\mathcal{C}_6$ and $\mathcal{C}_7$, Example 2}\label{fig.eg2}
			\end{center}
		\end{figure}	
		
		\noindent  The errors and orders of convergence for the displacement and the Airy-stress function are presented in  Table \ref{table.eg2}.  The orders of convergence results are similar to those obtained in Example 1. Note that Examples 1 and 2 are similar except in the sign of the term $|\boldsymbol{x}|^4$ that appears in the obstacle function.

		\medskip
		
		\noindent {\bf Example 3. (Violation of smallness assumption)}	It is interesting to observe that for $\chi$ and $\O$ from Example 1 with the source term $f=(x+3)^2(x-3)^2(y+3)^2(y-3)^2$, the primal dual active set algorithm is not convergent in 100 iterations of the algorithm. Consider $w(x,y)=(x+0.5)^2(y+0.5)^2(0.5-x)^2(0.5-y)^2 \in H^2_0(\O)$. Then $\frac{\|w\|_{L^2(\O)}}{\trinl w\trinr}=0.0278$ and  $\frac{\|w\|_{L^\infty(\O)}}{\trinl w\trinr}=0.0683$. Since $C_{\rm F}$ (resp. $C_{\rm S}$) is the supremum of  $\frac{\|z\|_{L^2(\O)}}{\trinl z\trinr}$ (resp. $\frac{\|z\|_{L^\infty(\O)}}{\trinl z\trinr}$) for all $z \in H^2_0(\O)$, this implies $C_{\rm F}\ge 0.0278$ (resp. $C_{\rm S}\ge 0.0683$). Use the definition of $M_{}(f,\chi)$ to obtain $C_{S}M_{}(f,\chi)\ge 20.7972$. Therefore the sufficient condition in Theorem \ref{thm.err} is violated.
		
		\smallskip
		\noindent
		For Example 1 with obstacle $\chi$ replaced by $\lambda \chi$, where $\lambda \in \mathbb{R}$, we noticed that the algorithm fails to converge for $\lambda\ge 4$ on $\cT_4$ and $\cT_5$. This illustrates the requirement of smallness assumption on the obstacle for optimal convergence rate.
		%
		%
		%
		\subsection{The von K\'{a}rm\'{a}n obstacle problem on the L-shaped domain}\label{eg.Lshaped}
		\noindent Consider L-shaped domain $\Omega=(-0.5,0.5)^2 \setminus[0,0.5]^2$, $f=0$ and
		$$\chi(\boldsymbol{x})=1-\frac{(x+0.25)^2}{0.2^2}-\frac{y^2}{0.35^2}$$
		as in \cite{brenner2013morley}. Choosing the initial mesh size as $h=0.7071$, the successive red-refinement algorithm computes $\T_1,\dots,\T_5$.
		
		%
		\begin{table}[h!!]   
			\caption{\small{Convergence results for the L-shaped domain}}
			{\small{\scriptsize
					\begin{center}
						\begin{tabular}{ ||c|c||c|c||c|c||c | c || c|c||c|c||}
							\hline
							$\ell$ &$h$     & $  \widetilde{e}_\ell(u)$ & EOC  & $  \widetilde{e}_\ell(v)$ & EOC  &$e_\ell(u)$ & EOC   &$e_\ell(v)$ & EOC  \\	
							\hline
							1&0.3536&0.046700&0.8276&0.141271& 1.8003&23.203954&0.7177& 2.260261&0.9584\\
							2&0.1768&0.021021&0.7196&0.056794& 1.9621&18.313668&0.8431&1.530842&1.0905\\	3&0.0884&0.025796&1.2271&0.017919& 2.1111&11.746209&0.9442& 0.761967& 1.1324\\
							4&0.0442&0.014152&1.5879&0.004655& 2.2774& 6.556709&1.0473& 0.352575&1.1531\\ 
							5&0.0221&0.004708&-&0.000960&-&3.172522&-&0.158538&-\\
							\hline									
							
						\end{tabular}
					\end{center}
				}}\label{table.eg3}
			\end{table}
			\begin{figure}[h!!]
				\begin{center}
					\begin{minipage}[H]{0.4\linewidth}
						{\includegraphics[width=12.65cm]{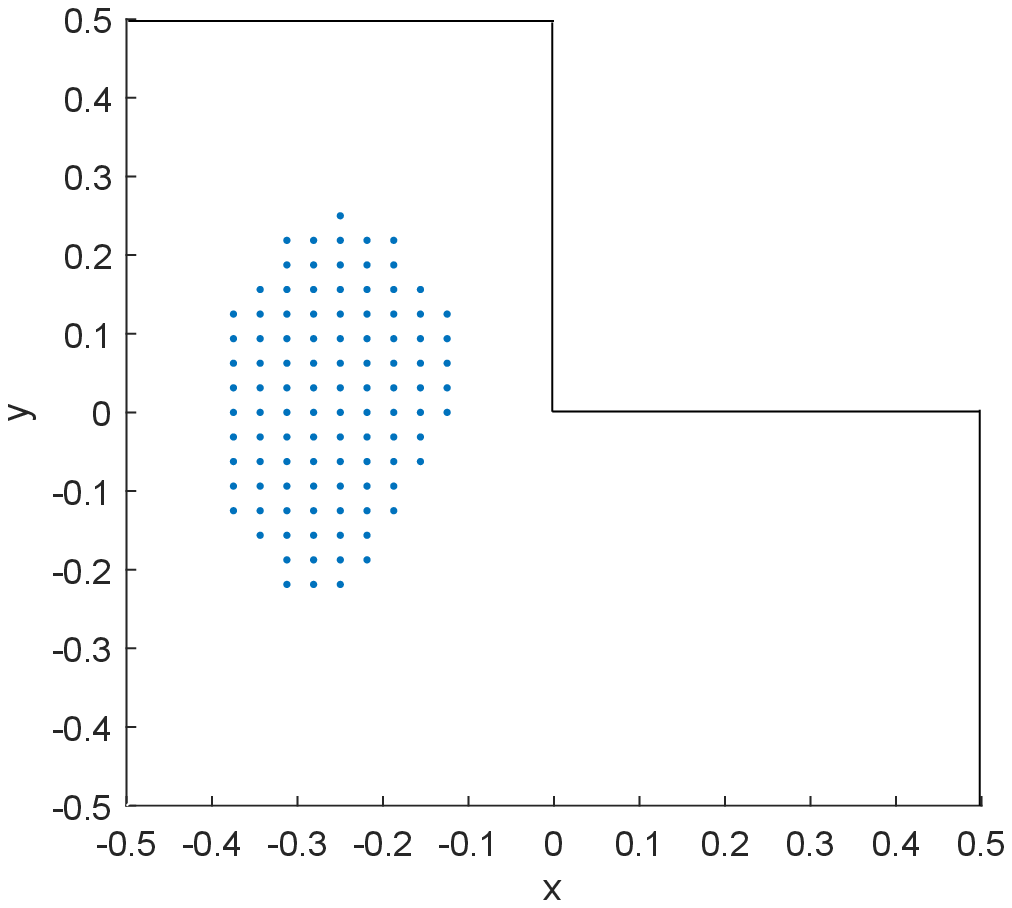}}
					\end{minipage}
					\begin{minipage}[H]{0.4\linewidth}
						{\includegraphics[width=13cm]{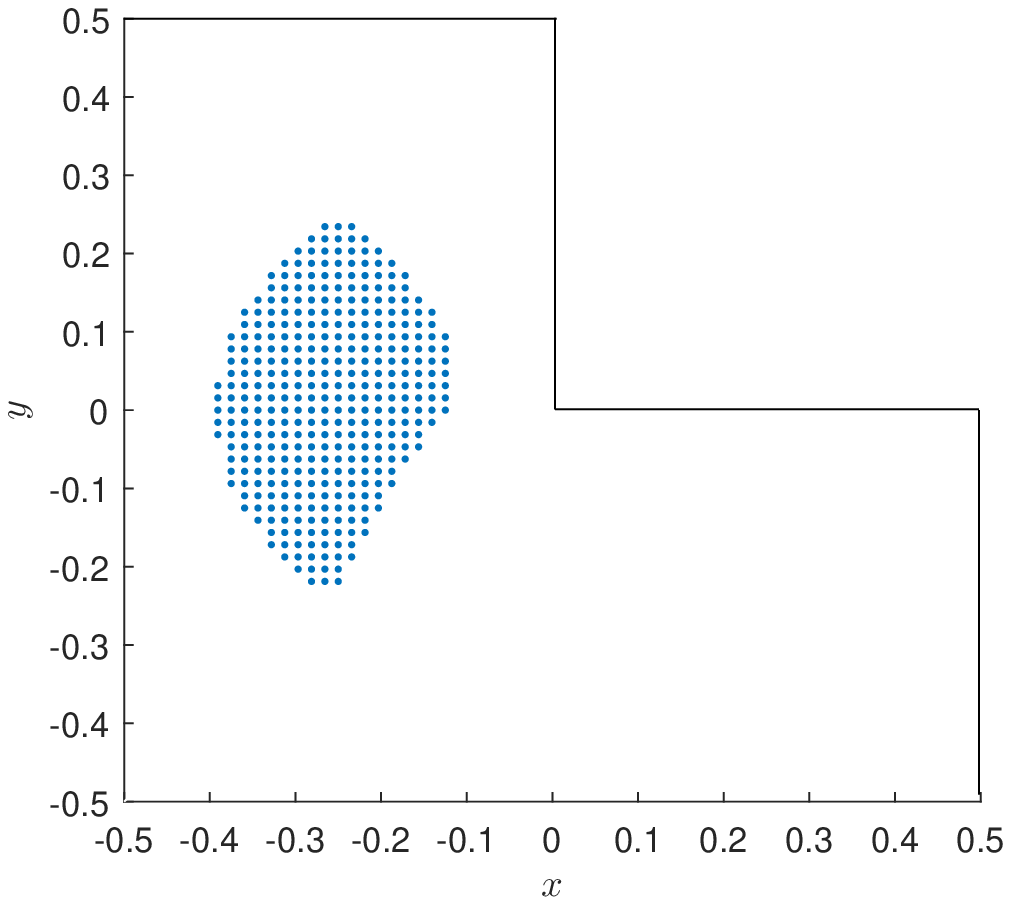}}
					\end{minipage}
					\caption{$\mathcal{C}_5$ and $\mathcal{C}_6$}\label{fig.eg3}
				\end{center}
			\end{figure}
			\noindent  Since $\O$ is non-convex (reduced elliptic regularity $\alpha = 0.5445$, \cite[Example 4]{brenner2013morley}), we expect only sub-optimal order of convergences in  energy norm and $L^\infty$ norm, that is, $\mathcal{O}(h^\alpha)$ convergence rate in the energy norm (see, Theorem \ref{thm.err}). However, linear order of convergence is preserved in the energy norm which indicates that the numerical performance is carried out in the non-asymptotic region. The discrete coincidence sets for last two levels are depicted in Figure \ref{fig.eg3}. The non-coincidence set is connected, which agrees with the result in \cite{caffarelli1979obstacle} since $\Delta^2\chi=0$ in $\O$ in this example. 	\smallskip
			
			\noindent The convergence rates in Table \ref{table.eg3} are not in direct contradiction to Theorem \ref{thm.err} but the reduced elliptic regularity suggests a lower rate $\alpha = 0.5445$ for L-shaped domain. A similar observation is in \cite[Table 5.5]{brenner2013morley} with orders of convergence $\approx$ 0.8 (resp. 1) for energy (resp. $L^\infty$) norm. In \cite{brenner2013morley}, the numbers are computed with the alternative definitions for error $e_\ell(u):=\trinl u_{\ell-1}-u_\ell \trinr_{\pw}$  (resp. $\widetilde{e}_\ell(u):=\max_{p \in \mathcal{V}_{\ell-1}} |u_{\ell-1}(p)-u_\ell(p)|$) and order of convergence ${\rm{EOC(\ell)}}:={\rm{log}}\big({e_{\ell-1}(u)}/{e_{\ell}(u)}\big)/{\rm{log}}(2)$       (resp. ${\rm{log}}\big({\widetilde{e}_{\ell-1}(u)}/{\widetilde{e}_{\ell}(u)}\big)/{\rm{log}}(2)$). With these definitions, undisplayed numerical experiments confirm the numbers displayed in \cite[Table 5.5]{brenner2013morley} precise up to the last digit. This numerical experiment suggests that our implementation is at least consistent with the one in \cite{brenner2013morley}. One possible explanation is that the corner singularity affects the asymptotic convergence rate for very small mesh-sizes only. This is known, for instance, for the L-shaped domain and the Poisson model problem with constant right hand side in the Courant ($\mathcal{P}_1$ conforming) finite element method. The expected rate 2/3 is visible only beyond $2\times 10^6$ triangles with far better empirical convergence rates before that. 
			\subsection{Conclusions}	\label{conclusion}
			The numerical results for the Morley FEM in the von K\'{a}rm\'{a}n obstacle problem are presented for square domain and L-shaped domain in Sections \ref{eg.square} and \ref{eg.Lshaped}. The outputs obtained for the square domain confirm the theoretical rates of convergence given in Theorem \ref{thm.err} for $\alpha=1$. Example 3 in Section \ref{eg.Lshaped} illustrates the requirement of smallness assumption on the obstacle for optimal convergence rate. 
			For the L-shaped domain, we expect reduced convergence rates in energy and $L^\infty$ norms from the elliptic regularity. However, linear order of convergence is preserved in the energy norm which indicates that the numerical performance is carried out in the non-asymptotic region.

			\medskip
			\noindent
			{\bf Acknowledgements.} The authors thankfully acknowledge the support from the MHRD SPARC project (ID 235) titled "The Mathematics and Computation of Plates" and the second to fifth authors also thank the hospitality of the Humboldt-Universit\"at zu Berlin for the corresponding periods 1st June 2019-31st August 2019 (second and fifth authors), 24th June 2019-30th June 2019 (third author), and 1st July 2019-31st July 2019 (fourth author). 
			
			\bibliographystyle{plain}
			\bibliography{vKeBib}

\end{document}